\documentclass[12pt]{amsart}
\usepackage{amssymb,latexsym, amsmath, amsxtra}
\usepackage{graphicx}
\usepackage{subfig}
\usepackage{xcolor}

\newtheorem{theorem}{Theorem}
\newtheorem{lemma}{Lemma}
\newtheorem{prop}{Proposition}
\newtheorem{corollary}{Corollary}
\newtheorem{conjecture}{Conjecture}
\newtheorem{remark}{Remark}

\begin{document}

	\title{Averages of quadratic twists of long Dirichlet polynomials}
	
	\author{Brian Conrey}
	\address{American Institute of Mathematics, 600 East Brokaw Rd, San Jose, CA 95112, USA and School of Mathematics, University of Bristol, Bristol BS8 1TW, UK}
	\email{conrey@aimath.org}
	\author{Brad Rodgers}
	\address{Queen's University, Kingston, ON Canada K7L 3N6}
	\email{brad.w.rodgers@gmail.com}
	
	\thanks{This project started at an AIM Square in 2019. We thank the members of that SQuaRE, 
		Alexandra Florea, Jon Keating, Julio Andrade, Chantal David and Matilde Lalin for 
		useful conversations about this work, and a referee for comments and corrections. We gratefully acknowledge support of the NSF FRG grant DMS 1854398. BR also acknowledges support of an NSERC grant.}
	
	\date{\today}
	
	\begin{abstract} 
		We investigate 
		averages of long Dirichlet polynomials twisted by Kronecker symbols
		and we compare our result with the recipe of Conrey-Farmer-Keating-Rubinstein-Snaith. 
		We are able to compute these averages in the case that the length of the polynomial is a power less than 2 of the basic scaling parameter on the assumption of the Lindel\"of Hypothesis for $L$-functions of quadratic characters, and we show that the answer is consistent with this recipe. This corresponds, in terms of the recipe, to verifying 0- and 1-swap terms.
	\end{abstract}

	\maketitle
	
	\section{Introduction}
	\label{sec:intro}
	
	\subsection{Main results}
	\label{subsec:mainresults}
	
	The purpose of this note is to estimate the averages of long Dirichlet polynomials twisted by Kronecker symbols, a topic closely connected to moments of $L$-functions over a quadratic family. Our main result is Theorem \ref{thm:main} and gives a power saving asymptotic estimate for
	\begin{equation}
		\label{eq:quad_sum}
		\sum_{d} \Psi(d/D) P_A(\chi_d)
	\end{equation}
	where the sum over $d$ is over a family of primitive quadratic characters with conductor $d$ and
	$$
	P_A(\chi_d) = \sum_n W(n/N) \frac{\tau_A(n) \chi_d(n)}{\sqrt{n}}
	$$
	is a Dirichlet polynomial of length roughly $N$ and $\tau_A$ is a weighted multiple divisor function (see \eqref{eq:tau_A_def} below). Here $\Psi$ and $W$ are smooth weights. We obtain an answer for 
	$$
	N = D^\eta\quad \textrm{for}\; 1 < \eta < 2.
	$$ 
	In fact we will also compute averages which have been twisted in the sum over $d$. As we will see the answer changes depending on the parameter $\eta$ and can be seen as a confirmation of the recipe of \cite{CFKRS} in a restricted range. We will make use of the methods in \cite{Sou}, in particular a critical role will be played by a Poisson summation formula proved there and recalled in Section \ref{subsec:poisson} of this paper.
	
	\subsection{Background}
	\label{subsec:background}
	
	The recipe of \cite{CFKRS} gives very precise conjectures for moments of products of shifted L-functions averaged over a natural family. 
	The so-called shifts are represented by a set $A=\{\alpha_1,\dots,\alpha_k\}$ of complex numbers with small real parts (in the case of orthogonal or symplectic families) and by two such sets $A$ and $B$ in the
	case of a unitary family.  
	The structure of the recipe is an expression of these shifted moments as a sum over ``swapped'' terms where the swaps are a shorthand way of 
	describing the replacement of  
	subsets of $A$ by all of the negatives of the elements of the subset (in the case of orthogonal or symplectic) or replacing  each subset of $A$ by
	all of the negatives of the elements of an equal sized subset of $B$ in the case of a unitary family.
	The size of the subset used for swapping is somehow  an indication of the complexity 
	of the terms, or at least the difficulty in explicating those terms. So, the 0-swap terms (with the empty subset of $A$ or of $A$ and $B$)
	represent the most obvious contributions to the moment, namely the diagonal, which basically arise from the $m=n$ term 
	in a unitary family, the $m=\square$ term in a family where the characters are quadratic characters, or the $m=1$ in a family of 
	cusp forms where the harmonic detector is the Petersson formula. (In this case, it depends on how the moment is attacked.)
	The 1-swap terms (where the swapping sets have cardinality 1) are harder to uncover. They arise from Soundararajan's Poisson formula in the case of
	quadratic characters (see \cite{Sou} and \cite{Flo1});  they arise from the asymptotic large sieve (see \cite{CoIwSou,ChLi} and also \cite{RoSou});
	they arise from the Kloosterman sum terms in the Petersson formula in the case of averaging cusp forms (see \cite{KoMiVa}); 
	and they arise from an assumption about divisor correlations in the case of the Riemann zeta-function (see \cite{CoGh,CoGo})
	
	As already mentioned, one is interested in averaging  products of shifted L-functions; see for instance \cite{AnKe, Flo1, Flo2, Flo3, DiWh, Di, Son, Dj} for work on families of quadratic $L$-functions in particular.
	However, only low moments seem possible to consider using current techniques. 
	Another possible way to make progress is to consider an average of a Dirichlet polynomial which truncates the Dirichlet series for a product of shifted L-functions.
	Indeed, it is exactly because low moments themselves can be closely approximated by relatively short Dirichlet polynomials of this sort that we are able to compute them.
	
	In computing averages of Dirichlet polynomials as opposed to moments, we will see below that comparison with \emph{portions} of the moment recipe are possible even for high moments. It may be necessary to assume the Lindel\"{o}f Hypothesis to bound 
	the error terms in this case. Such a  use of Lindel\"{o}f may seem contrary to the long-term goals of proving the moment conjectures; but at this early stage of development it is important to see what is true in order to be guided how to proceed. 
	The family under study will have a basic parameter which is essentially the analytic conductor of the family. The truncation parameter,  called the ``length'' of the polynomial, is the power of the analytic conductor.  In general, if the length is $<1$, only the diagonal or 0-swap terms will contribute.
	
	If the length is  between 1 and 2, then we are in the situation described in the previous paragraph. It is this situation we will examine for the family of quadratic characters, confirming the appearance of 0- and 1-swap terms in the aforementioned recipe. In recent work, Hamieh and Ng have used this perspective and Hardy-Littlewood type conjectures for divisor functions to study mean values of Dirichlet polynomials averaged in the $t$-aspect \cite{HaNg}. Likewise in work concurrent with this paper, Baluyot and Turnage-Butterbaugh have pursued a similar strategy for Dirichlet polynomials averaged over all primitive characters using the asymptotic large sieve \cite{BaTu}, and Conrey and Fazzari have studied the analogous problem for long Dirichlet polynomials with modular coefficients using Kloosterman sums and the Petersson formula \cite{CoFaz}. In an appendix to her thesis \cite{Flo1} Florea has used closely related ideas to study off-diagonal terms arising from an application of Soundararajan's method to the problem of moments of quadratic L-functions in function fields, likewise making a comparison between the polar structure of these terms and that of 1-swap terms from the CFKRS recipe.
	
	In forthcoming work Baluyot and Conrey use the results in this paper to heuristically study swaps larger than $0-$ or $1-$ over the family of quadratic $L$-functions, a symplectic analogue of heuristics in \cite{BaCo,CoKeI, CoKeII, CoKeIII, CoKeIV, CoKeV} for a unitary family.
	
	Let us also note that while in press, \v{C}ech has shown in \cite{Ce} that Theorem \ref{thm:main} of this paper can also be elegantly deduced from the theory of multiple Dirichlet series, and a function field version of Conjecture \ref{conj: recipe} has now been proved in work of Bergström-Diaconu-Petersen-Westerland and Miller-Patzt-Petersen-Randal-Williams \cite{BeDiPeWe, MiPaPeRaWi}.

	\section{Setup and results}
	\label{sec:setup}
	
	\subsection{Long Dirichlet polynomials}
	\label{sec:long_dirichlet}
	
	We consider averages of long Dirichlet polynomials over a family of quadratic characters and twisted by a quadratic character. Set
	\begin{eqnarray*}
		\mathcal{S}_A(D;\, \ell):= \sum_{d\in\mathcal D}\Psi(d/D)\sum_{n\geq 1}W(n/N)  \frac{\tau_A(n) \chi_d(n\ell)}{\sqrt{n}},
	\end{eqnarray*}
	with $N = D^\eta$ for $1 < \eta < 2$,
	and where $A$ is a (multi)set of complex numbers and $\tau_A$ is defined implicitly by 
	\begin{equation}
		\label{eq:tau_A_def}
		Z_A(s)=\prod_{\alpha \in A} \zeta(s+\alpha) =\sum_{n=1}^\infty \frac{\tau_A(n)}{n^s}.
	\end{equation}
	Here $\mathcal D$ is a family of primitive quadratic Dirichlet characters (i.e. Kronecker symbols, or characters of quadratic field extensions of $\mathbb Q$). We will say more about our choice of family below.
	The test functions $\Psi$ and $W$  are $C^\infty$ functions on $\mathbb R^+$ which can be thought of as approximating the characteristic function of $[1,2]$ and of $(0,1]$ respectively. Finally $\ell$ is a positive integer which is not too large (indeed the reader may wish to set $\ell = 1$ on first reading).
	
	\begin{remark}
	In \eqref{eq:tau_A_def} and in what follows, we adopt the convention that if the (multi)set $A$ has elements listed multiply, they appear multiple times in the product. Thus if $A = \{\alpha_1,...,\alpha_k\}$ our convention will be that $\prod_{\alpha \in A} \zeta(s+\alpha) = \prod_{i=1}^k \zeta(s+\alpha_i)$. In general we allow the collections of complex numbers we consider to be multisets (but continue to refer to them as sets where there is no chance of confusion).
	\end{remark}

	Our goal is to express $\mathcal S_A(D;\ell)$  in terms of the 0- and 1-swap terms predicted by the recipe of CFKRS; we explain this prediction below.
	
	\subsection{Quadratic characters and L-functions}
	\label{subsec:quadratic_char}
	
	Recall that a quadratic character modulo $q$ is a non-trivial character $\chi$ such that $\chi^2$ is a trivial character modulo $q$. We are interested in primitive quadratic characters (in this case a character modulo $q$ has conductor $q$) and these have the following description (see \cite[Ch. 9]{MV} for proofs of these claims): 
	\begin{itemize}
		\item for any odd prime $p$ there is a unique primitive character of conductor $p$ given by 
		\begin{equation}
			\label{eq:quadchar1}
			n \mapsto \left( \frac n p \right)
		\end{equation}
		where $( \tfrac \cdot \cdot )$ is the Legendre symbol, and there are no primitive quadratic characters of conductor $p^j$ for $j \geq 2$.
		
		\item There is no primitive quadratic character of conductor $2$. There is one primitive quadratic character of conductor $4$ given by
		\begin{equation}
			\label{eq:quadchar2}
			n \mapsto 
			\begin{cases} 
				+1 & \mbox{if $n\equiv 1 \bmod 4$}\\
				-1 & \mbox{if $n\equiv 3 \bmod 4$}\\
				0 & \mbox{if $n$ is even.}
			\end{cases}
		\end{equation}
		
		\item There are two primitive quadratic characters of conductor $8$, given by the functions
		\begin{equation}
			\label{eq:quadchar3}
			n \mapsto 
			\begin{cases} 
				+1 & \mbox{if $n\equiv \pm 1 \bmod 8$}\\
				-1 & \mbox{if $n\equiv \pm 3 \bmod 8$}\\
				0 & \mbox{if $n$ is even,}
			\end{cases}
		\end{equation}
		and
		\begin{equation}
			\label{eq:quadchar4}
			n \mapsto 
			\begin{cases} 
				+1 & \mbox{if $n\equiv 1,3 \bmod 8$}\\
				-1 & \mbox{if $n\equiv 5,7 \bmod 8$}\\
				0 & \mbox{if $n$ is even.}
			\end{cases}
		\end{equation}
		\item For $j \geq 4$, there are no primitive quadratic characters of conductor $2^j$.
	\end{itemize}
	
	Furthermore if $q = q_1 q_2$ with $q_1$ and $q_2$ coprime and $\chi_1$ and $\chi_2$ primitive quadratic characters of conductors $q_1$ and $q_2$ respectively, then $\chi_1\chi_2$ is a primitive quadratic character of conductor $q$ and moreover all primitive quadratic characters of conductor $q = q_1 q_2$ can be factored this way. Therefore a complete list of primitive quadratic characters is formed by multiplying various combinations of the characters described in \eqref{eq:quadchar1}-\eqref{eq:quadchar4} together.
	
	There is a convenient parameterization of such characters in terms of the Kronecker symbol (also written $(\tfrac \cdot \cdot )$ and which coincides with the Legendre symbol where the latter is defined). Recall that an integer $d$ is said to be a fundamental discriminant if $d=1$ or $d$ is the discriminant of a quadratic field. It is known that an integer $d$ is a fundamental discriminant if and only if it is either a squarefree integer congruent to $1$ modulo $4$, or $4$ times a squarefree integer congruent to $2$ or $3$ modulo $4$. (See e.g. \cite[Ch. 13.1]{IrRo}.)
	
	We have the following parametrization of primitive quadratic characters. For $d$ a fundamental discriminant set and $d\neq 1$ set
	$$
	\chi_d(n) = \left( \frac d n \right).
	$$
	If $d = 1$ set $\chi_1(n) = 1$. $\chi_d$ is a primitive quadratic character modulo $|d|$, and moreover one may show that all primitive real characters are of this form. (Thus is may be seen that the character in \eqref{eq:quadchar1} is $\chi_p$ if $p \equiv 1 \bmod 4$ and $\chi_{-p}$ if $p \equiv 3 \bmod 4$; the character in \eqref{eq:quadchar2} is $\chi_{-4}$; the character in \eqref{eq:quadchar3} is $\chi_8$; and the character in \eqref{eq:quadchar4} is $\chi_{-8}$.)
	
	A critical point is that
	$$
	\chi_d(-1) = \begin{cases}
		+1 & \mbox{if $d > 0$} \\
		-1 & \mbox{if $d < 0$}.
	\end{cases}
	$$
	In other words $\chi_d$ is an even character if $d > 0$ and $\chi_d$ is an odd character if $d < 0$. 
	
	Each fundamental discriminant $d$ is the discriminant of the field $\mathbb{Q}(\sqrt{d})$, and we have the Dedekind zeta-function of a quadratic field is given by
	$$
	\zeta_{\mathbb{Q}(\sqrt{d})} = \zeta(s) L(s,\chi_d)
	$$
	where $\zeta(s)$ is the Riemann zeta-function and 
	$$
	L(s,\chi_d) = \sum_{n=1}^\infty \frac{\chi_d(n)}{n^s}.
	$$
	These $L$-functions have functional equations which depend on the conductor of the characters and on the parity of the character. The functional equation is given by
	$$
	L(1-s,\chi_d) = |d|^{s-1/2} X_+(1-s) L(s,\chi_d)
	$$
	for even characters (i.e. $d > 0$) and
	$$
	L(1-s,\chi_d) = |d|^{s-1/2} X_-(1-s) L(s,\chi_d)
	$$
	for odd characters (i.e. $d < 0$) where
	$$
	X_+(1-s) = 2^{1-s} \pi^{-s} \cos(\tfrac{\pi s}2) \Gamma(s)
	$$
	$$
	X_-(1-s) = 2^{1-s} \pi^{-s} \sin(\tfrac{\pi s}2) \Gamma(s).
	$$
	We will write
	$$
	X_d(1-s) = \begin{cases}
		|d|^{s-1/2} X_+(1-s) & \mbox{if $d > 0$}\\
		|d|^{s-1/2} X_-(1-s) & \mbox{if $d < 0$}.
	\end{cases}
	$$
	From Stirling's formula, the reader can verify that
	\begin{equation}
		\label{eq:stirling}
		|X_d(s)| \asymp (|d|(|t|+2))^{1/2-\sigma},
	\end{equation}
	for $s = \sigma + it$ in the vertical strip $1/4 \leq \sigma \leq 3/4$ say.

	\subsection{The recipe and Dirichlet polynomials}
	\label{subsec:the_recipe}
	
	We are interested in averaging products of these (shifted) L-functions at the central point 1/2.
	The recipe of CFKRS gives a conjectural answer to such an average  over a family
	of real characters.  
	The families we consider  are  all of the 
	fundamental discriminants  $d$ that are $\equiv a \bmod q$ for some fixed $a$ and $q$;  i.e. they are all the fundamental discriminants contained in a fixed arithmetic progression.  
	Let $\mathcal D$ denote such a family. Let $D$ be a (large) parameter and let $A$ be a set of complex numbers with  real parts  $\ll \frac{1}{\log D}$  and  imaginary parts
	$\ll D$. Let 
	$$\mathcal {L}_A(s, \chi_d)=\prod_{\alpha\in A} L(s+\alpha,\chi_d)=\sum_{n=1}^\infty \frac{\tau_A(n)\chi_d(n)}{n^s}.$$
	Let $\Psi  \in C^\infty[1,2]$ be a smooth function supported on the interval $[1,2]$. We consider the average
	$$ \mathcal M_A(D;\ell):= \sum_{d\in \mathcal D} \Psi\left(\frac dD\right) \mathcal L_A(1/2,\chi_d) \chi_d(\ell).$$
	The conjectural answer for this average depends on the auxiliary function
	$$\mathcal B(A):=\sum_{ n=\square}^\infty \frac{\tau_A(n)}{\sqrt{n}}$$
	and, more precisely, on 
	$$
	\mathcal B^{(d)}(A;\ell):=\sum_{(n,d)=1\atop n\ell=\square}^\infty \frac{\tau_A(n)}{\sqrt{n}}.
	$$
	$\mathcal{B}(A)$ and $\mathcal{B}^{(d)}(A;\ell)$ are defined by these sums for $\Re\, \alpha_i  > 0$ for all $i$ and are defined by analytic continuation for $\Re\, \alpha_i > -1/4$ for all $i$ (see Lemma \ref{lem:B_cont} below).
	
	\begin{conjecture}(Recipe)
		\label{conj: recipe}
		Given a set $A$ of complex numbers with real parts $\ll 1/\log D$ and imaginary parts $\ll D$, there is a $\delta>0$ (depending on the number of elements in the set $A$) such that 
		\begin{eqnarray*}
			\mathcal M_A(D;\ell)= \sum_{d\in \mathcal D\atop (d,\ell)=1} \Psi\left(\frac dD\right) \sum_{U\subset A} \prod_{u\in U} X_d(1/2 +u) \mathcal B^{(d)}(A-U+U^-;\ell)+O(\ell^{1/2}D^{1-\delta})
		\end{eqnarray*}
		where $A-U+U^-$ is the set $\{\alpha\in A: \alpha\notin U\}\cup \{ -u:u\in U\}$.
	\end{conjecture}
	
	\begin{remark}
	We adopt the notational convention that the sum over $U$ above is a sum over all subsets of $A$, including the empty set. In the case that $U = \emptyset$, the summand becomes $\mathcal B^{(d)}(A;\ell)$. 
	\end{remark}
	
	\begin{remark}
		It is important to note that this statement of the recipe for this family differs in a subtle but important way from previous statements. In previous incarnations,
		the fundamental Dirichlet series had the form
		$$ \sum_{n=1\atop n=\square}^\infty \frac{\tau_A(n) \prod_{p\mid n} \frac{p}{p+1}} {n}$$
		where the factor $\prod_{p\mid n}\frac {p}{p+1}$ is part of the ``arithmetical'' apparatus of the recipe. It turns out that the recipe can be recast in the manner above with this factor $\prod_{p\mid n} \frac{p}{p+1}$  becoming part of the ``averaging over $d$'' portion of the recipe  and the condition $(d,n)=1$
		being inserted into the fundamental Dirichlet series as above.  See Lemma 2 and also Lemma 7 and the subsequent paragraph for some calculations which show that the statements
		of the recipe  are equivalent, though perhaps psychologically different. 
	\end{remark}
	
	We apply Conjecture \ref{conj: recipe} to the average $\mathcal{S}_A(D;\ell)$ by using the Mellin transform
	$$
	\breve{F}(s):= \int_0^\infty F(u) u^{s-1}\, du.
	$$
	Note that if $F$ is a smooth function compactly supported on a closed interval $I \subset (0,\infty)$, we have that $\breve{F}$ is an entire function of exponential type with
	$$
	\breve{F}(a+it) \ll_{a,A} \frac{1}{|t|^A},
	$$
	for any $a\in \mathbb{R}$ and any constant $A > 0$, and we have the inversion formula
	$$
	F(u) = \frac{1}{2\pi i} \int_{(a)} \breve{F}(s) u^{-s}\, ds
	$$
	where the path of integration is on any vertical line segment from $a-i\infty$ to $a + i\infty$.
	
	We use the notation
	$$ A_s=\{\alpha+s:\; \alpha \in A\}$$
	in what follows. Conjecture \ref{conj: recipe} then suggests
	
	\begin{conjecture}
		\label{conj: dirichpoly}
		Let $a$ be a small fixed positive constant (e.g. $a = 1/10$). Let $H$ be any fixed positive number. Given a set $A$ of complex numbers with real parts $\ll 1/\log D$ and imaginary parts $\ll D$, there is a $\delta > 0$ (depending on the number of elements in $A$) such that
		\begin{eqnarray*}
			\mathcal S_{A}(D;\ell)&=&\frac{1}{2\pi i}\int_{(a)}\breve{W}(s) N^s
			\sum_{d\in \mathcal D\atop (d,\ell)=1} \Psi\left(\frac dD\right) \sum_{U\subset A_s} \prod_{u\in U} X_d(1/2 +u) \mathcal B^{(d)}(A_s-U+U^-;\ell)~ds\\
			&&\qquad  +O(\ell^{1/2}D^{1-\delta}),
		\end{eqnarray*}
		for $N =  D^\eta$ with $0 < \eta < H$.
	\end{conjecture}
	
	Indeed, it should be possible to demonstrate that Conjectures \ref{conj: recipe} and \ref{conj: dirichpoly} are equivalent -- Conjecture \ref{conj: dirichpoly} following from Conjecture \ref{conj: recipe} via a simple Mellin transformation, and the other direction via an approximation of $L$-functions with Dirichlet polynomials -- but we do not pursue a formal equivalence here.

	\subsection{Statement of results}
	\label{subsec:statement_of_results}
	
	In this paper we will focus on the family $\mathcal D$ of Soundarajan's paper \cite{Sou} on quadratic non-vanishing, namely the family of positive discriminants divisible which are  divisible by $8$. This family $\mathcal{D}$ can be written $\mathcal{D} = \{8d:\; 2d \; \textrm{is squarefree and positive}\}$.  The reason for choosing this family is only ease of explication.
	For $8d \in \mathcal{D}$, define
	$$P_A(\chi_{8d};\ell) := \sum_{n}W\left(\frac n N\right)\frac{\tau_A(n) \chi_{8d}(n\ell)}{\sqrt{n}}$$
	where $\ell$ is odd. (Because $\chi_{8d}(m) = 0$ if $m$ is even, we have that if $\ell$ is even, then $P_A(\chi_{8d};\ell) = 0$. For the same reason the sum above can be restricted to odd $n$.)
	
	Let
	$$\mathcal S_A(D;\ell):=\sum_{d}\mu^2(2d) \Psi\left(\frac d {D}\right) P_A(\chi_{8d};\ell).$$
	Our main result is the following:
	\begin{theorem}
		\label{thm:main} 
		Assume the Lindel\"{o}f Hypothesis for the functions $L(s,\chi_d)$. Let $a$ be a small fixed constant (e.g. $a = 1/10)$. Given a set $A$ of complex numbers with positive real parts $\ll 1/\log D$ and imaginary parts $\ll 1$, we have for any $\epsilon > 0$, 
		\begin{eqnarray*}  \mathcal S_A(D;\ell)
			&=&
			\sum_{(d,\ell)=1 } \mu^2(2d) \Psi\left(\frac d{D}\right) 
			\frac{1}{2\pi i}\int_{(a)}\breve W(s) N^s 
			\sum_{U\subset A_s\atop |U|\le 1} \prod_{u\in U} X_{8d}(1/2 +u) \mathcal B^{(2d)}(A_s-U+U^-;\ell) ~ds\\
			&&\qquad +O(N^{1/4}D^{1/2+\epsilon} \ell^{1/4+\epsilon}).
		\end{eqnarray*}
	\end{theorem}
	
	In this theorem we have adopted the convention that $D$ is sufficiently large, so that the shifts of size $\ll 1/\log D$ will be smaller than given small constants like $a$. The rest of the paper should be read with this convention in mind. It should be possible with more book-keeping to prove this theorem where shifts have imaginary part $\ll D$, but we have focused on the range above for ease of explication.
	
	One expects that $P_A(\chi_{8d}; \ell)$ is of size $O(N^{\epsilon})$ so that one expects $S_A(D; \ell)$ to be of size $O(D N^\epsilon)$. Thus the error term is a power savings for $N = D^\eta$ as long as $\eta < 2$. For $\eta > 2$ the error term will be larger than the main term.
	
	\begin{remark}
	In Theorem \ref{thm:main}, we have restricted to a set $A$ of complex numbers with real parts which are positive to allow for a proof which is slightly easier to keep track of. It is likely that with a little more bookkeeping the proof will allow for a set $A$ of complex numbers with real parts of arbitrary sign (still with real parts $\ll 1/\log D$ and imaginary parts $\ll 1$). Indeed such a result was proved using different methods while this paper was in press in \cite{Ce}.
	\end{remark}
	
	Note that the sum over $U$ in this theorem is restricted  to $|U|\le 1$.  As the $\eta$ in $N=D^\eta$ gets larger we would expect more terms to enter into the formula; namely all the terms $U$ with $|U| \leq \eta$. Indeed, one may show that terms with $|U| > \eta$ will make no contribution if added in:
	
	\begin{prop}
		\label{prop:swap_dropping}
		Let $a$ be a small fixed constant (e.g. $a = 1/10$). For any fixed positive integer $j$ and any fixed positive number $\eta < j$, consider $D \geq 1$ and $N = D^\eta$ and a set $A$ of complex numbers as in Theorem \ref{thm:main}. If $d \asymp D$, then
		\begin{equation}
			\label{eq:swap_drop}
			\frac{1}{2\pi i} \int_{(a)} \breve{W}(s) N^s \sum_{U \subset A_s \atop |U| = j} \prod_{u\in U} X_d(1/2+u) \mathcal{B}^{(d)}(A_s - U + U^{-};\ell)\,ds \ll D^{-\epsilon}
		\end{equation}
		for some constant $\epsilon > 0$ (depending on $j$ and $\eta$).
	\end{prop}
	
	Thus combining this Proposition with the Theorem above, we see the Lindel\"of Hypothesis implies Conjecture \ref{conj: dirichpoly} for the family of Soundararajan when $\eta < 2$. We give the proof of Proposition \ref{prop:swap_dropping} in Section \ref{subsec:B_cont} below.
	
	\section{Preliminaries} 
	\label{sec:prelim}
	
	\subsection{Analytic continuation of $\mathcal{B}$}
	\label{subsec:B_cont}
	
	We have already used the following analytic continuation of the function $\mathcal{B}^{(d)}(A;\ell)$ in order to state our results:
	
	\begin{lemma}
		\label{lem:B_cont}
		Fix $\delta > 0$ and use the notation $A = \{\alpha_1,...,\alpha_k\}$. For any $d$, $\ell$ there exists a function $\mathcal{C}^{(d)}(A;\ell)$ given by an Euler product converging whenever $\Re \alpha_i \geq -1/4+\delta$ for all $i$, such that in the region $\Re\; \alpha_i > 0$ for all $i$ one can write,
		$$
		\mathcal{B}^{(d)}(A;\ell) = \prod_{\alpha \in A} \zeta(1+2\alpha) \prod_{\{\alpha,\beta\} \subset A} \zeta(1+\alpha+\beta)\, \cdot \, \mathcal{C}^{(d)}(A;\ell).
		$$
		
		Moreover $\mathcal{C}^{(d)}(A;\ell) \ll_{\delta,\epsilon} d^\epsilon$ for any $\epsilon > 0$ in the region $\Re \alpha_i \geq -1/4+\delta$ for all $i$, uniformly in $d$ and $\ell$.
	\end{lemma}
	
	\begin{remark}
	To avoid any chance for confusion regarding the product indexed by subsets of the multiset $A$, let us note that this could also be written
	$$
	\mathcal{B}^{(d)}(A;\ell) = \prod_{i} \zeta(1+2\alpha_i) \prod_{i < j} \zeta(1+\alpha_i+\alpha_j)\, \cdot \, \mathcal{C}^{(d)}(A;\ell).
	$$
	\end{remark}
	
	Note that this continuation implies in Conjecture \ref{conj: dirichpoly} and Theorem \ref{thm:main} that the integral over the line with real part $a$ will not involve integrating over a singularity for sufficiently large $D$.
	
	\begin{proof}
		For $\Re  \alpha_i > 0$ for all $i$, one finds
		$$
		\mathcal{B}^{(d)}(A; \ell) = \prod_p \Big(1 + \frac{\tau_A(p^2)}{p} + \frac{\tau_A(p^4)}{p^2} + \cdots \Big) \, \cdot \mathcal{E}^{d}(A;\ell),
		$$
		with
		$$
		\mathcal{E}^{(d)}(A;\ell) = \prod_{\substack{\nu_p(\ell) \; \textrm{odd} \\ (p,d)=1}} \Big(\frac{\tau_A(p)/p^{1/2} + \tau_A(p^2)/p^{3/2} + \cdots}{1 + \tau_A(p^2)/p + \cdots}\Big) \prod_{p|d} \frac{1}{1 + \tau_A(p^2)/p+\cdots}.
		$$
		$\mathcal{E}^{(d)}(A;\ell)$ is a finite Euler product, and it is easy to verify that the first term of it is bounded in the region $\Re\, \alpha_i \geq -1/4+\delta$ for all $i$, while the second term is $\ll_{\delta, \epsilon} d^{\epsilon}$. Hence
		$$
		\mathcal{E}^{(d)}(A;\ell) \ll_{\epsilon, \delta} d^\epsilon.
		$$
		Yet using
		$$
		\tau_A(p^2) = \sum_{\alpha \in A} p^{-2\alpha} + \sum_{\{\alpha,\beta\} \subset A} p^{-\alpha} p^{-\beta},
		$$
		we have
		$$
		\prod_p \Big(1 + \frac{\tau_A(p^2)}{p} + \frac{\tau_A(p^4)}{p^2} + \cdots \Big) = \prod_{\alpha \in A} \zeta(1+2\alpha) \prod_{\{\alpha,\beta\} \subset A} \zeta(1+\alpha+\beta)\, \cdot \, \mathcal{E}'(A),
		$$
		with $\mathcal{E}'(A) = O_\delta(1)$ in the region $\Re\, \alpha_i \geq -1/4+\delta$ for all $i$. Hence
		$$
		\mathcal{B}^{(d)}(A;\ell) = \prod_{\alpha \in A} \zeta(1+2\alpha) \prod_{\{\alpha,\beta\} \subset A} \zeta(1+\alpha+\beta) \cdot \mathcal{E}'(A) \cdot \mathcal{E}^{(d)}(A;\ell)
		$$
		and the claim follows.
	\end{proof}
	
	\begin{proof} [Proof of Proposition \ref{prop:swap_dropping}]
		We claim for $s = \sigma + it$,
		\begin{equation}
			\label{eq:lswap_bound}
			\sum_{U \subset A_s \atop |U| = j} \prod_{u\in U} X_d(1/2+u) \mathcal{B}^{(d)}(A_s - U + U^{-};\ell) \ll D^{-k\sigma+\epsilon} (|t|+2)^{-k\sigma}.
		\end{equation}
		To see this we use Corollary \ref{cor:doublepermutationsum}, proved in the Appendix, with 
		$$
		F(u_1,...,u_j; v_1,...,v_{k-j}) = X_d(1/2+u_1) \cdots X_2(1/2+u_j) \mathcal{B}^{(d)}(-u_1,...,-u_j,v_1,...,v_{k-j})\Delta(u;v),
		$$
		in the notation of that Corollary. If $s = \sigma + it$, and $u_i$ and $v_i$ all have imaginary parts within a bounded distance of $t$ and real parts within a sufficiently small bounded distance of $\sigma$, we have that $F(u;v) \ll D^{-j\sigma+\epsilon} (|t|+2)^{-j\sigma+\epsilon}$, from Stirling's formula \eqref{eq:stirling}, Lemma \ref{lem:B_cont} and the Lindel\"of Hypothesis. Hence the left hand side of \eqref{eq:lswap_bound} is
		$$
		\frac{1}{j! (k-j)!} \frac{1}{(2\pi i)^k} \int_\Gamma \cdots \int_\Gamma  \int_\Gamma \cdots \int_\Gamma \frac{F(z;w) \Delta(z;w) \Delta(z)^2 \Delta(w)^2}{\Delta(z;A_s) \Delta(w;A_s)} \, d^{k-j}w \, d^j z, 
		$$
		where we choose a contour $\Gamma$ of bounded length enclosing all points of $A_s$ which remains a small distance away from any point of $A_s$. This then yields \eqref{eq:lswap_bound}.
		
		By using the decay of the function $\breve{W}$, we thus see that \eqref{eq:swap_drop} is $\ll D^{\eta a} D^{-j a} D^{\epsilon}$ for arbitrary $\epsilon > 0$. As $\eta < j$ this proves the claim.
	\end{proof}
	
	\subsection{Sieving for squarefrees and Soundararajan's summation formula}
	\label{subsec:poisson}
	
	By  the Lindel\"{o}f Hypothesis, for $d \asymp D$ and shifts $\alpha \in A$ with real part $\ll 1/\log D$ and imaginary part $\ll D$,
	\begin{eqnarray*}
		P_A(\chi_{8d};\ell) &=&\frac{\chi_{8d}(\ell)}{2\pi i} \int_{(1)} \mathcal L_A(s+1/2,\chi_{8d}) \breve W(s) N^s  ~ds\ll  D^\epsilon N^\epsilon
	\end{eqnarray*}
	upon  moving the path of integration to the $\Re s=\epsilon $ line.  The same estimate holds whether $d$ is a fundamental 
	discriminant or not. 
	Since
	$$\mu^2(d)=\sum_{c^2\mid d}\mu(c),$$
	we have 
	$$\mathcal S_{A}(D;\ell)=\sum_c \sum_{\substack{d  ~ \textrm{odd} \\ c^2|d}} \mu(c) \Psi\left(\frac {d} D\right) P_A(\chi_{8d};\ell) .$$
	We introduce a parameter $Y$ and split this sum into those terms with $c\le Y$ and those with $c>Y$. An optimal choice of $Y$ will be made later. The sum with the terms with $c>Y$ are, by the Lindel\"{o}f Hypothesis,
	\begin{eqnarray*}
		\ll \sum_{c>Y} \sum_{b \ll D/c^2} |P_A(\chi_{8b c^2};\ell)|\ll  D^{1+\epsilon} Y^{-1},
	\end{eqnarray*}
	as long as $N$ grows no more than polynomially in $D$.
	
	Thus as long as $N$ grows no more than polynomially in $D$ we have 
	\begin{eqnarray*}
		\mathcal S_A(D;\ell)= \sum_{d  ~ \textrm{odd}} \sum_{c^2\mid d \atop c\le Y }\mu(c) \Psi\left(\frac {d} D\right) P_{A}(\chi_{8d};\ell) +O(D^{1+\epsilon} Y^{-1}).
	\end{eqnarray*}
	Our basic tool to analyze this sum is the Poisson formula of Soundararajan's paper \cite[Lem 2.6]{Sou}
	
	\begin{lemma}  \label{lemma:poisson} (Poisson summation formula)
		Let $m$ be a positive odd integer. We have 
		\begin{eqnarray*}
			\sum_{d ~ \textrm{odd} } \sum_{c^2\mid d\atop c\le Y} \mu(c)
			\Psi\left(\frac d D\right )\chi_d(m)
			=D\frac{\left(\frac 2 m \right)}{2m}\sum_{(c,2m)=1\atop c\le Y} \frac{\mu(c)}{c^2} \sum_{k = -\infty}^\infty (-1)^k   G_k(m) \tilde{\Psi}\left(\frac{kD}{2c^2m}\right)
		\end{eqnarray*} 
		where
		$$\tilde {\Psi}(x)=\int_{-\infty}^\infty \Psi(u)(\cos 2\pi xu+\sin 2\pi xu)~du$$
		and where $G_0(m)=\phi(m) $ if $m=\square$ and 0 otherwise; and  if $k\ne 0$ then  $G_k(m)$  is a multiplicative function of $m$ whose value at 
		$m=p^\mu$ is determined by the value of $\kappa$ for which
		$p^\kappa\mid\mid k$ by 
		\begin{eqnarray*}
			G_k(p^\mu)=\left\{
			\begin{array}{ll}
				0 & \mbox{if $ \mu \le \kappa$ and $\mu$ is odd} \\
				\phi(p^\mu) & \mbox{if $ \mu \le \kappa$ and $\mu$ is even} \\
				-p^\kappa & \mbox{if $ \mu = \kappa+1 $ and $\mu$ is even} \\
				\left(\frac{kp^{-\kappa}}{p}\right)p^\kappa\sqrt{p} & \mbox{if $ \mu = \kappa+1 $ and $\mu$ is odd} \\
				0 & \mbox{if $\mu\ge \kappa+2$}
			\end{array}
			\right.
		\end{eqnarray*}
	\end{lemma}
	
	In fact $G_k(m)$ can also be written in terms of Gauss sums (see \cite[Sec. 2.2]{Sou}), but it is the multiplicative characterization above that will be useful for us.
	
	We apply this to our sum with $m=n\ell$ (noting that both $n$ and $\ell$ must be odd if $\chi_{8d}(n\ell) \neq 0$) and get
	$$
	\mathcal S_A(D;\ell) = \mathcal S'_A(D;\ell) + O(D^{1+\epsilon}Y^{-1}),
	$$
	where
	\begin{equation}
		\label{eq:SA after poisson}
		\mathcal S'_A(D;\ell)=D \sum_{(n,2)=1} W\left(\frac n N\right) \frac{\tau_A(n)}{\sqrt{n}}\frac{1}{2n\ell}
		\sum_{(c,2n\ell)=1\atop c\le Y}\frac{\mu(c)}{c^2}\sum_{k = -\infty}^\infty (-1)^k G_k(n\ell) \tilde \Psi\left(\frac {kD}{2c^2n\ell}\right).
	\end{equation}
	We have used the fact that $\chi_{8d}(n\ell) = \left( \frac{8d}{n\ell}\right) = \left( \frac 8 {n\ell}\right) \chi_d(n\ell)$, the first factor of which cancels the term $\left( \frac 2 {n\ell} \right)$ coming from the right hand side of Lemma \ref{lemma:poisson}.
	
	\subsection{An outline of the proof}
	
	Having approximated $\mathcal S_A(D;\ell)$ by the expression in \eqref{eq:SA after poisson}, we can now outline the main steps in our proof to be carried out in the following sections. In Section \ref{subsec:S0} we examine those terms from the right hand side of \eqref{eq:SA after poisson} corresponding to $k=0$. The contribution from these terms will be diagonal contributions and will be relatively easy to evaluate. These terms will be seen to correspond to the contribution in Theorem \ref{thm:main} from $U = \emptyset$ (the 0-swap contribution). In Section \ref{subsec:Sneq0} the terms for which $k\neq 0$ are examined. The contribution from these terms will be found via the collection of residues from a multiple contour integral. Terms for which $k$ is a square will be seen to contribute to the main term, while other terms will be seen to contribute only to an error term. The function $G_{k^2}(n)$ is a multiplicative function of two variables, and this fact, combined with a factorization into Euler products and analytic continuation, allows for contours to be shifted and residues to be computed. Finally in Section \ref{sec:identities} the answer that is obtained this way is compared with the recipe prediction for $1$-swap terms in Theorem \ref{thm:main}. That the two expressions are equal is not immediately clear; it requires a somewhat involved comparison of Euler products and gamma factors.
	
	\section{Diagonal and off-diagonal terms}
	\label{sec:diagonal and offdiagonal}
	
	\subsection{Evaluating $\mathcal{S}^0$}
	\label{subsec:S0}
	
	The contribution from the $k=0$ term of $\mathcal S'_A(d;\ell)$ in \eqref{eq:SA after poisson} is 
	\begin{eqnarray*}
		\mathcal S_A^0(D;\ell):=D \tilde{\Psi}(0)\sum_{ \substack{ n\ell= \square  \\ (n,2)=1 }}W\left(\frac nN\right)  
		\sum_{(c,2n\ell)=1\atop c\le Y}\frac{\mu(c)}{c^2} \frac{\phi(n\ell)}{2n\ell}\frac{\tau_A(n)}{\sqrt{n}}.
	\end{eqnarray*}
	
	In this section we show that $\mathcal S_0(A;\ell)$ is equal to the 0-swap terms in Theorem \ref{thm:main} up to a small error term. 
	
	\begin{lemma} \label{lemma:small mobius sum} For any $m$ and $D$ and a smooth function $\Psi$ we have
		$$D \tilde{\Psi}(0) \frac{\phi(m)}{m}  \sum_{{c\le Y\atop (c,m)=1} }\frac{\mu(c)}{c^2} = \sum_{(d,m)=1}\mu^2(d) \Psi\left(\frac d D\right)+O\left(D^{1/2}\log Y + \frac{D}{Y}\right).$$
	\end{lemma}
	
	\begin{proof}
		Decomposing $\mu^2(d) = \sum_{bc^2 = d} \mu(c)$ and splitting into parts with $c \leq Y$ and $c > Y$,
		\begin{eqnarray*}
			\sum_{(d,m)=1}\mu^2(d)\Psi\left(\frac d D\right)&=& \sum_{(b c^2,m)=1 \atop c \le Y} \mu(c)\Psi(b c^2/D) + O\Big(\sum_{(b c^2,m)=1 \atop c > Y} |\Psi(b c^2/D)| \Big) \\
			&=&  \sum_{(c,m)=1\atop c\le Y } \mu(c) \sum_{(b,m)=1} \Psi(b c^2/D) +O\Big( \sum_{c > Y} D/c^2 \Big)\\
			&=&  \sum_{(c,m)=1\atop c\le Y } \mu(c) \sum_{g\mid m}\mu(g) \sum_{h} \Psi(gh c^2/D) +O\left(\frac DY\right).
		\end{eqnarray*}
		The sum over $h$ is 
		\begin{eqnarray*}
			\frac{1}{2\pi i} \int_{(2)}\zeta(s)  \breve{\Psi}(s) \frac{D^s}{c^{2s}g^s} ~ds =\breve{\Psi}(1)\frac{D}{c^2g} +O\left( D^{1/2}c^{-1}g^{-1/2} \right)
		\end{eqnarray*}
		This gives 
		\begin{eqnarray*}
			\sum_{(d,m)=1}\mu^2(d)\Psi\left(\frac d D\right)
			&=&  \breve{\Psi}(1)\frac{\phi(m)}{m} D \sum_{{c\le Y\atop (c,m)=1} }\frac{\mu(c)}{c^2} +O\left(D^{1/2}\log Y + \frac DY\right)\\
		\end{eqnarray*}
		Since
		$\breve \Psi(1) =  \tilde\Psi(0)$, the proof is complete.
	\end{proof}
	
	We apply this with $m=2n\ell$ with $n $ and $\ell$ odd. Noting that $\phi(2n\ell) =\phi(n\ell)$ we see that (for $a = 1/10$ say)
	\begin{equation}
		\label{eq:1swap}
		\mathcal S_A^0(D;\ell)=
		\frac{1}{2\pi i}\int_{(a)} \breve W(s) N^s
		\sum_{\atop (d,\ell)=1} \mu^2(2d) \Psi\left(\frac dD\right) \mathcal B^{(2d)}(A_s;\ell)~ds +O( D^{1/2}\log Y +DY^{-1})
	\end{equation}
	and so contributes the 0-swap terms (i.e. those with $|U|=0$) in Theorem \ref{thm:main} with an error term that will be acceptable with a good choice of $Y$.
	
	\subsection{Evaluating $\mathcal{S}^{\neq 0}$}
	\label{subsec:Sneq0}
	
	Now we treat $k\ne 0$. This computation is somewhat lengthy and consists of four steps. We outline the steps here, but a reader may prefer to continue reading on and then come back to this outline while organizing the argument in their mind. In \textbf{Step 1} we express $\mathcal S_A^{\neq 0}$ as a multiple contour integral in variables $s$ and $w$. This contour integral involves a function $U^\ast(s,w;Y)$, with a sum of terms $U_c(s,w)$ for $c \leq Y$. In \textbf{Step 2} we analytically continue the functions $U_c(s,w)$ in the variable $w$; this allows us to push back the contour in $w$ and extract residues. In \textbf{Step 3}, we evaluate the residues which remain. This is done via a different analytic continuation in $s$, which allows us both to push backward the contour in $s$ and then to complete the sum in $c$ up to a small error term. In \textbf{Step 4} we use the functional equation of the zeta function to rewrite this completed sum in a way that looks more like terms from the $1$-swap recipe.
	
	In this way we obtain a good approximation for $\mathcal S_A^{\neq 0}$ which (at least superficially) resembles the $1$-swap terms in Theorem \ref{thm:main}. What remains is to show that the terms that arise this way actually are equal to those that appear in Theorem \ref{thm:main}. This (nontrivial) comparison is done in the following section.

	\textbf{Step 1.} We use
	\begin{eqnarray*}
		\sum_n a_n g(n)=\frac{1}{2\pi i} \int_{(a)} \sum_{n=1}^\infty \frac{a_n}{n^w} \left(\int_0^\infty g(t) t^{w-1}~dt\right) dw \end{eqnarray*}
	for an appropriate value of $a$. 
	
	Let $\mathcal S_A^{\ne 0}(D;\ell)$ be the contribution of the non-zero $k$ terms in $\mathcal S'_A(D;\ell)$. We have
	\begin{align*}
		&\mathcal S_A^{\ne 0}(D;\ell) \\ &=\frac{D}{2\pi i} \int_{(\frac 12)} \sum_{(n,2)=1} n^{-w} \frac{\tau_A(n)}{\sqrt{n}}\frac{1}{2n\ell}
		\sum_{(c,2n\ell)=1\atop c\le Y}\frac{\mu(c)}{c^2}\sum_{k\ne 0}  (-1)^k G_k(n\ell)  \int_0^\infty  W\left(\frac t N\right)\tilde \Psi\left(\frac {kD}{2c^2t\ell}\right)
		t^{w-1}~dt ~dw
	\end{align*}
	
	We rewrite the integral over $t$  above using the following lemma.
	\begin{lemma} \label{lemma:rewritten integral} 
		For any $\gamma$ with $\Re w< \gamma $ we have
		\begin{eqnarray*}
			\int_0^\infty W\left(\frac t N\right) \tilde{\Psi}\left(\frac{C}{ t}\right) t^{w-1}~dt
			= \frac1{2\pi i} \int_{(\gamma)} \breve{W}(s)\breve{\Psi}(1-s+w) X(1-s+w)N^sC^{w-s}
			~ds
		\end{eqnarray*}
		where 
		$$X(1-s)=(2\pi)^{-s}\Gamma(s)\left( \cos\frac{\pi s}{2}+\sin\frac{\pi s}{2}\right).$$ 
	\end{lemma}
	
	\begin{proof}
		The integral in question is 
		\begin{eqnarray*}
			\int_0^\infty  W\left(\frac t N\right) t^{w-1} \int_{-\infty}^\infty \Psi(u) (\cos 2\pi u \frac{C}{ t} +\sin 2\pi u \frac{C}{ t}) ~du ~dt
		\end{eqnarray*}
		We substitute $v=\frac{C}{ t} u$ and have (recall $\Psi$ is supported on $[1,2]$)
		\begin{eqnarray*}
			\int_0^\infty  W\left(\frac t N\right) t^{w-1} \int_0^\infty \Psi\left(\frac{t v}{C}\right) (\cos 2\pi v +\sin 2\pi v)
			\frac{ t}{C} ~dv ~dt
		\end{eqnarray*}
		We use 
		$$  W\left(\frac t N\right)=\frac{1}{2\pi i}\int_{( \gamma )} \breve W(s) \left(\frac tN\right)^{-s} ~ds$$
		where we initially choose $\gamma < 1+\Re w$
		and bring  the $t$-integral to the inside; it  is
		$$\int_0^\infty \Psi\left(\frac{ t v}{C}\right) t^{w-s+1}\frac{dt}{t}=
		\breve\Psi(w-s+1) \left(\frac{C}{ v}\right) ^{w-s+1}.$$
		The $v$-integral is 
		$$\int_0^\infty v^ {s-w-1}(\cos 2\pi v+\sin 2\pi z) ~dv=  X(1-s+w).$$
		The lemma now follows upon collecting these results with a larger range of $\gamma$ permitted by analytic continuation. 
	\end{proof}

	Let
	\begin{eqnarray*}
		U^\ast(s,w;Y):=  \sum_{(c,2\ell)=1\atop c\le Y}  \frac{\mu(c)}{c^{2-2s+2w}}U^\ast_c(s,w)
	\end{eqnarray*}
	where
	\begin{eqnarray*}
		U^\ast_c(s,w):= \sum_{(n,2c)=1}  \frac{\tau_A(n)}{n^{1+w}} 
		\sum_{k\ne 0}  (-1)^k \frac{G_k(n\ell)/\sqrt{n}}{ k^{s-w}}
	\end{eqnarray*}
	Then  
	\begin{eqnarray*}
		\mathcal S_A^{\ne 0}(D;\ell) = \frac{1}{(2\pi i)^2} \int_{(\frac 12 )}   \int_{( 2 )} \breve{W}(s)\breve{\Psi}(1-s+w) X(1-s+w)N^s  U^\ast(s,w;Y) \left( \frac{D}{2\ell}\right)^{1-s+w}
		~ds~dw 
	\end{eqnarray*}
	
	\textbf{Step 2.} Now we begin the process of analytic continuation of $ U^\ast_c(s,w)$ in $s$ and $w$. In this step, we analytically continue in $w$, and shift the contour in $w$ to extract residues. The residues which are extracted will correspond to square values of $k$ in the sums above.
	
	Note that  $G_{k}(n)=G_{4k}(n)$ for odd $n$. For any non-zero integer $k$  we can write $4k=k_1 k_2^2$ uniquely where $k_1$ is a fundamental discriminant ($k_1=1$ is allowed) and where $k_2$ is positive. 
	\begin{lemma} \label{lemma:tau G continuation}
		We have
		\begin{eqnarray*} 
			\sum_{(n,2c)=1} \frac{\tau_A(n) G_{4k}(n\ell)/\sqrt{n}}{n^{1+w}}
			=\mathcal L_A(1+w,\chi_{k_1})  \mathcal{G}_A(1+w;k,\ell,c)
		\end{eqnarray*}
		where
		$ \mathcal{G}_A(w+1;k,\ell,c)$ is holomorphic for $\Re w>-1/2+\max_{\alpha \in A}|\Re \alpha|$ and is $\ll c^\epsilon k^\epsilon \ell^{\frac12+\epsilon}(\ell,k_2^2)^{1/2}$ there, for any $\epsilon > 0$.
	\end{lemma}
	The proof is essentially the same as that of Lemma 5.3 of \cite{Sou}.
	
	So now 
	$$  U^\ast_c(s,w)=   \sum_{k\ne 0}\frac{(-1)^k }{k^{s-w}} \mathcal L_A(w+1,\chi_{k_1})  \mathcal{G}_A(w+1; k,\ell,c).$$
	
	Now $\mathcal L_A(w+1,\chi_{k_1})$ is entire unless
	$k_1=1$ in which case $\mathcal L_A(w+1,\chi_{1}) = Z_A(w+1)$ which has poles at $w=-\alpha$ for $\alpha \in A$. 
	We move the paths of integration now, first we move $w$ to the line $\Re w=-1/2+\delta$ and then we move $s$ to $\Re s=1/2 +2\delta$.  
	
	In doing so we cross poles at $w=-\alpha$ when $k_1=1$ (i.e when $k=\square$).  Choose $\epsilon > 0$ and then $\delta >0$ sufficiently small based on $\epsilon$. By Lemma 2 and the Lindel\"{o}f hypothesis the integral on the new path is $\ll N^{1/2}D^{\epsilon}Y^{1+\epsilon}\ell^{\frac 12+2\epsilon} $, where again we suppose $N$ grows no more than polynomially in $D$.
	
	Thus,
	$$ \mathcal S_A^{\ne 0}(D;\ell) = \mathcal S_A^{\square}(D;\ell) +O(N^{1/2}D^{\epsilon}Y^{1+\epsilon}\ell^{\frac 12+\epsilon})$$
	where
	\begin{eqnarray} \label{eqn:S1}  \mathcal S_A^{\square}(D;\ell) &:=& \nonumber
		\sum_{\alpha\in A}  \frac 1{2\pi i} \int_{( \frac 12 +\delta)}    \breve{W}(s)\breve{\Psi}(1-s-\alpha) X(1-s-\alpha)N^s \\&& \qquad 
		\times   \operatornamewithlimits{Res}_{w=-\alpha} U(s,w;Y)  \left( \frac{D}{2\ell}\right)^{1-s-\alpha}
		~ds.
	\end{eqnarray}
	where
	\begin{eqnarray*}
		U(s,w;Y):= 
		\sum_{(c,2\ell)=1\atop c\le Y}\frac{\mu(c)}{c^{2-2s+2w}} U_c(s,w),\end{eqnarray*}
	with 
	$$
	U_c(s,w) := \sum_{k=1}^\infty (-1)^{k^2}{k^{2(s-w)}} Z_A(w+1)\mathcal{G}_A(w+1;k,\ell,c)
	$$
	defined for $\Re w > -1/2+\delta$ and $\Re(s-w) > 1/2$.
	
	Now consider the region $\Re w > \max_{\alpha \in A} |\Re \alpha|$ and $\Re (s-w) > 1/2$. In this region we may expand this expression as a Dirichlet series and we have
	\begin{align} 
		\label{eqn:euler2}
		\notag U_c(s,w) &= \sum_{(n,2c)=1} \frac{\tau_A(n)}{n^{1+w}}   
		\sum_{k=1}^\infty  (-1)^k \frac{G_{k^2}(n\ell) /\sqrt{n}}{k^{2s-2w}}\\
		\notag  &   = -(1-2^{1+2w-2s}) \sum_{(n,2c)=1} \frac{\tau_A(n)}{n^{1+w}} 
		\sum_{k=1}^\infty   \frac{G_{k^2}(n\ell) /\sqrt{n}}{k^{2s-2w}} \\
		&= - \frac{1-2^{1+2w-2s}}{1-2^{2w-2s}} \sum_{(n,2c)=1} \frac{\tau_A(n)}{n^{1+w}} 
		\sum_{k \; \mathrm{odd}}  \frac{G_{k^2}(n\ell) /\sqrt{n}}{k^{2s-2w}},
	\end{align}
	where in the last two lines we have again used the property $G_k(m)=G_{4k}(m)$ as long as $m$ is odd. Note that in this range of $w$ and $s$, both the sum over $n$ and the sum over $k$ will be absolutely convergent.  
	
	\textbf{Step 3.} We now analytically continue the expression $\operatornamewithlimits{Res}_{w = -\alpha} U_c(s,w)$ in $s$. This will allow us analytically continue $\operatornamewithlimits{Res}_{w = -\alpha} U(s,w;Y)$ as well, which in turn allows us to shift the contour of $s$ in \eqref{eqn:S1} to a vertical line with small real part. This enables us to complete the sum over $c$ in $U(s,w;Y)$. Because the resulting expression is a complete sum in all parameters, it will be easier to treat using Euler products in the next section.
	
	It will be convenient to isolate an Euler product with only odd prime factors. For that reason we define the functions
	\begin{equation}
		\label{eq:odd_zeta}
		\zeta^{[2]}(s) = (1-2^{-s})\zeta(s), \quad Z^{[2]}_A(s) = \prod_{\alpha \in A} \zeta^{[2]}(\alpha+s),
	\end{equation}
	which are defined so as for $\Re s > 1$ to have the usual Euler products with the local factor at the prime $2$ removed.
	
	\begin{lemma}
		\label{lemma:singularity location}
		Let $a$ be a small fixed constant. Then in the region $\Re w > 0$ and $\Re(s-w) > 1/2$, for odd coprime $c$ and $\ell$,
		\begin{equation}
			\label{eq:Uc(s,w) continuation}
			U_c(s,w) = -\frac{1-2^{1+2w-2s}}{1-2^{2w-2s}} Z^{[2]}_A(1+w) \zeta^{[2]}(2s-2w) V_c(s,w;\ell),
		\end{equation}
		where $V_c(s,w;\ell)$ is analytic in $s$ and $w$ for $a \leq \Re s$ and $-a/4 \leq \Re w \leq \Re s -a/2$ and in this region (i) has a convergent Euler product over only odd primes, and (ii) satisfies $V_c(s,w;\ell) \ll_\epsilon c^\epsilon \ell^{1+\epsilon}$, where $\epsilon > 0$ can be chosen arbitrarily.
		
		Consequently for odd coprime $c$ and $\ell$, the function $U_c(s,w)$ has an analytic continuation to this region and
		\begin{equation}
			\label{eq:ResUc(s,w) continuation}
			\operatornamewithlimits{Res}_{w = -\alpha} U_c(s,w) = -\frac{1}{2} \cdot \frac{1-2^{1-2\alpha-2s}}{1-2^{-2\alpha-2s}} Z^{[2]}_{A'}(1-\alpha) \zeta^{[2]}(2s+2\alpha) V_c(s,-\alpha;\ell),
		\end{equation}
		for $A':= A \setminus \{\alpha\}$.
	\end{lemma}
	
	\begin{proof}
		Note that $f(k,n)=G_{k^2}(n)$ is a multiplicative function of two variables; i.e. we have $f(k_1k_2,n_1n_2)=f(k_1,n_1)f(k_2,n_2)$ whenever $(k_1n_1,k_2n_2)=1$. 
		This allows us to express $U_c(s,w)$ as an Euler product. Let
		$$
		\mathcal{H}_p(s,w;\ell) = \sum_{n,k=0}^\infty \sum_{k,n=0}^\infty \frac{\tau_A(p^n)G_{p^{2k}}(p^{n+\nu_p(\ell)})}{p^{nw+n/2+2k (s-w)}},
		$$
		where $\nu_p(\ell) := \max \{k:\; p^k | \ell\}$ is the $p$-adic valuation of $\ell$. We calculate from \eqref{eqn:euler2} that for $\Re w > 0$ and $\Re(s-w) > 1/2$, as long as $c$ and $\ell$ are coprime odd numbers,
		\begin{align}
			\label{eq:Uc(s,w) euler}
			\notag U_c(s,w) &= -\frac{1-2^{1+2w-2s}}{1-2^{2w-2s}} \prod_{p \nmid 2c\ell} \mathcal{H}_p(s,w;1) \prod_{p | c} \sum_{k=0}^\infty \frac{1}{p^{k(2s-2w)}} \prod_{p | \ell} \mathcal{H}_p(s,w;\ell) \\
			&= -\frac{1-2^{1+2w-2s}}{1-2^{2w-2s}}  \prod_{p>2} \mathcal{H}_p(s,w;1) \prod_{\substack{p|c}} \frac{ \sum_{k\geq 0} 1/p^{k(2s-2w)}}{\mathcal{H}_p(s,w;1) } \prod_{p | \ell} \frac{\mathcal{H}_p(s,w;\ell)}{\mathcal{H}_p(s,w;1)}
		\end{align}
		
		Using the values in Lemma \ref{lemma:poisson} to evaluate $G_{p^{2k}}(p^n)$, we find for $a \leq \Re s $ and $-a/4 \leq \Re w \leq \Re s - a/2$, if we let $v = 2s-2w$ (so $\Re v \geq a$),
		\begin{align*}
			\mathcal{H}_p(s,w;1) &=\Big(1 - \frac{1}{p^v}\Big)^{-1} \Big(1 + \frac{\tau_A(p^2) \phi(p^2)}{p^{3+2w} p^v} + \frac{\tau_A(p^4) \phi(p^4)}{p^{6+4w} p^{2v}} + \frac{\tau_A(p^6) \phi(p^6)}{p^{9+6w} p^{3v}}+\cdots \Big) \\
			&\quad\quad + \Big( \frac{\tau_A(p)}{p^{1+w}} + \frac{\tau_A(p^3)}{p^{2+2w}} + \frac{\tau_A(p^5)}{p^{3+5w}} +\cdots \Big)\\
			&= \Big(1 - \frac{1}{p^v}\Big)^{-1} \Big(1 + \frac{\tau_A(p)}{p^{1+w}} + \frac{\tau_A(p^2)}{p^{2(1+w)}}+\cdots\Big)\Big(1 + O\Big(\frac{1}{p^{1+\delta}}\Big)\Big),
		\end{align*}
		for sufficiently small $\delta > 0$ as long as $a$ is sufficiently small. From these terms the factor $Z^{[2]}_A(1+w) \zeta^{[2]}(2s-2w)$ can be extracted. 
		
		Note furthermore that this implies in this region of $w$ and $s$ that $\mathcal{H}_p(s,w;1) = e^{O_a(1)}$. We can use the same argument and the values in Lemma \ref{lemma:poisson} to see that $\mathcal{H}_p(s,w;\ell) \ll_a p^{\nu_p(\ell)}$. And finally it is easy to see that $\sum 1/p^{k(2s-2w)} = e^{O_a(1)}$. 
		
		If we return to \eqref{eq:Uc(s,w) euler} using this information and the fact that number of primes factors of $c$ is $o(\log c)$ (and likewise for $\ell$) we obtain \eqref{eq:Uc(s,w) continuation}.
		
		The claim \eqref{eq:ResUc(s,w) continuation} then follows immediately as $U_c(s,w)$ has a simple pole for $w$ at $-\alpha$, coming from the factor $\zeta^{[2]}(1+w+\alpha)$ in $Z^{[2]}_A(1+w)$. Note that $\zeta^{[2]}(1+w+\alpha)$ has a residue of $1/2$ at $w = -\alpha$.
	\end{proof}
	
	Return to \eqref{eqn:S1}. We may push the contour from the line $\Re s = 1/2+\delta$ to $\Re s = a'$ for $a'$ any small fixed positive constant, using \eqref{eq:ResUc(s,w) continuation}. Note that for all $c$, the function $\operatornamewithlimits{Res}_{w = -\alpha} U_c(s,w)$ has no poles for $s$ in between these lines; the simple pole of the function $\zeta(2s+2\alpha)$ at $2s+2\alpha = 1$ is canceled by the zero of $(1-2^{1-2\alpha-2s})$ at this point. Now bring the sum over $\alpha \in A$ inside the integral, and note from the singularity structure of \eqref{eq:Uc(s,w) continuation} we may write,
	$$
	\sum_{\alpha \in A} \operatornamewithlimits{Res}_{w = -\alpha} U_c(s,w) = \frac{1}{2\pi i} \int_{\gamma} U_c(s,w)\, dw,
	$$
	where $\gamma = \cup \gamma_\alpha$ is a union of sufficiently small contours encircling each $\alpha \in A$.
	
	Now along the line $\Re s = a'$ we complete the sum in $c$. By Lemma \ref{lemma:singularity location} and the above expression for $\sum \operatornamewithlimits{Res} U_c(s,w)$, one sees that this introduces an error term of size $\ll D^{1+\epsilon} \ell^\epsilon /Y^{1-\epsilon}$ as long as $a'$ is sufficiently small (depending on $\epsilon$).
	
	Hence from \eqref{eqn:S1}, what we have shown is that for sufficiently small positive $a'$ (depending on $\epsilon$),
	
	\begin{multline}
		\label{eq:S_on_line_a}
		\mathcal S_A^{\ne 0}(D;\ell) = \sum_{\alpha\in A}  \frac 1{2\pi i} \int_{(a')}    \breve{W}(s)\breve{\Psi}(1-s-\alpha) X(1-s-\alpha)N^s  \operatornamewithlimits{Res}_{w=-\alpha} U(s,w)  \left( \frac{D}{2\ell}\right)^{1-s-\alpha}
		~ds \\ +O(N^{1/2}D^{\epsilon}Y^{1+\epsilon}\ell^{\frac 12+\epsilon} + D^{1+\epsilon} \ell^\epsilon /Y^{1-\epsilon}),
	\end{multline}
	where $U(s,w) := \lim_{Y\rightarrow\infty} U(s,w;Y)$ is well-defined in the region of this integral.
	
	\textbf{Step 4.} We now rewrite the expression above in a way that more closely resembles the $1$-swap terms we are seeking.
	
	\begin{lemma}
		\label{lem:func_eq}
		For $\Re s = a'$ with $a' > 0$ sufficiently small and all elements of the set $A$ with sufficiently small imaginary parts, and odd $\ell$,
		\begin{multline*}
			X(1-s-\alpha) \operatornamewithlimits{Res}_{w=-\alpha} U(s,w) \frac{1}{(2\ell)^{1-s-\alpha}} \\
			= 8^{s+\alpha-1} \chi(1/2+s+\alpha) \frac{\ell^{s+\alpha-1}}{2} \zeta^{[2]}(1-2s-2\alpha) Z_{A'}^{[2]}(1-\alpha) \sum_{(c,2\ell)=1} \frac{\mu(c)}{c^{2-2s-2\alpha}}V_c(s,-\alpha;\ell).
		\end{multline*}
	\end{lemma}
	
	Here $\chi(s) := X_+(s)$, so as usual $\zeta(s) = \chi(s)\zeta(1-s).$
	
	\begin{proof}
		Note the easily proven identity 
		\begin{eqnarray} \label{eqn:identityX} X(1-s-\alpha) \chi(2s+2\alpha)=4^{-\alpha-s }\chi(1/2+s+ \alpha).\end{eqnarray}
		We will also use \eqref{eq:ResUc(s,w) continuation} of Lemma \ref{lemma:singularity location}. Note that in \eqref{eq:ResUc(s,w) continuation} we may make the substitution
		$$
		\zeta^{[2]}(2s+2\alpha) = \frac{1-2^{-2s-2\alpha}}{1-2^{2s+2\alpha-1}} \zeta^{[2]}(1-2s-2\alpha) \chi(2s+2\alpha).
		$$
		Applying \eqref{eqn:identityX} and a little algebra, we obtain lemma.
	\end{proof}
	
	Plainly the left hand side of this identity can be directly substituted into \eqref{eq:S_on_line_a}, and because the terms $8^{s+\alpha-1} \chi(1/2+s+\alpha)$ resemble those in the $1$-swap recipe, this brings us closer to a proof of Theorem \ref{thm:main}.
	
	\section{Identities}
	\label{sec:identities}
	
	The answer we have obtained has begun to look like the contribution from $1$-swap terms. There are two remaining ingredients to complete our proof. First, in the following subsection, we apply Lemma \ref{lem:func_eq} and make use of a comparison of Euler products to get a version of Theorem \ref{thm:main} with no averaging over $d$. Then, in second subsection, we show that an average in $d$ can be extracted from the result just obtained. This allows us to compare to the claim in Theorem \ref{thm:main}.
	
	\subsection{A comparison of Euler products}
	\label{sec:euler_compare}
	
	For $\Re \alpha_i > 0$ for all $i$, define
	$$
	\tilde{\mathcal{B}}^{(2)}(A;\ell):= \sum_{\substack{n\ell = \square \\ (n,2)=1}} \frac{\tau_A(n) a_{n\ell}}{\sqrt{n}}, \quad \textrm{with}\; a_m:= \prod_{p|m} \frac{p}{p+1},
	$$
	and define $\tilde{\mathcal{B}}^{(2)}(A;\ell)$ for other parameters $\alpha_i$ by analytic continuation. In analogy with Lemma \ref{lem:B_cont}, we have
	$$
	\tilde{\mathcal{B}}^{(2)}(A;\ell) = \prod_{\alpha \in A} \zeta(1+2\alpha) \prod_{\{\alpha,\beta\} \subset A} \zeta(1+\alpha+\beta)\, \cdot \, \tilde{\mathcal{C}}^{(2)}(A;\ell),
	$$
	for $\tilde{\mathcal{C}}^{(2)}(A;\ell)$ given by an Euler product convergent as long as $\Re \alpha_i > -1/4$ for all $i$ (with the proof of Lemma \ref{lem:B_cont} applying mutatis mutandis). Thus $\tilde{\mathcal{B}}^{(2)}(A;\ell)$ may be analytically continued into this region.
	
	In this subsection we prove the following identity.
	
	\begin{lemma}
		\label{lem:matching_eulers}
		For sufficiently small $a'> 0$ and $\Re s = a'$, for odd positive integers $\ell$,
		\begin{multline}
			\label{eq:matching_eulers}
			\frac{\ell^{s+\alpha-1}}{2} \zeta^{[2]}(1-2s-2\alpha) Z_{A'}^{[2]}(1-\alpha) \sum_{(c,2\ell)=1} \frac{\mu(c) V_c(s,-\alpha;\ell)}{c^{2-2s-2\alpha}} =  \frac{\lambda}{\zeta^{[2]}(2)} \tilde{\mathcal{B}}^{(2)}((A')_s \cup \{-\alpha-s\};\ell),
		\end{multline}
		where $\lambda := 1/2$. 
	\end{lemma}
	
	We have left $\lambda$ in this expression rather than writing $1/2$ in order to easily refer to this factor later. (Recall also the notation $A' = A \setminus \{\alpha\}$.)
	
	\begin{proof}
		Retracing our steps, note that the left hand side of \eqref{eq:matching_eulers} is equal to
		\begin{multline}
			\label{eq:back_to_U(s,w)}
			\ell^{s+\alpha-1}\frac{ \zeta^{[2]}(1-2s-2\alpha)}{\zeta^{[2]}(2s+2\alpha)} \Big[ \operatornamewithlimits{Res}_{w = -\alpha} U(s,w) \Big(-\frac{1-2^{2w-2s}}{1-2^{1+2w-2s}}\Big)\Big] \\
			= \ell^{s+\alpha-1}\frac{ \zeta^{[2]}(1-2s-2\alpha)}{\zeta^{[2]}(2s+2\alpha)} \Big[ \lim_{w = -\alpha} \frac{\lambda}{\zeta^{[2]}(1+w+\alpha)} U(s,w) \Big(-\frac{1-2^{2w-2s}}{1-2^{1+2w-2s}}\Big)\Big]
		\end{multline}
		Again $\lambda = 1/2$; this factor appears here because $1/\zeta(s) \sim \lambda/\zeta^{[2]}(s)$ as $s\rightarrow 1$.
		
		We show that the left hand side and right hand side of \eqref{eq:matching_eulers} are equal by an identification of Euler products. We will show for $p > 2$ and $\nu \geq 0$ even,
		\begin{multline} \label{eqn:identity3}
			\frac{1}{p^{\nu-\nu(s+ \alpha)}} \frac{1-p^{-2(s+\alpha)}}{1-p^{2(\alpha+s)-1}} \left(1-\frac 1 p\right)
			\sum_{c,k,n=0\atop \min\{c,n+\nu\}=0}^\infty \frac{\mu(p^c)\tau_{A_s}(p^n)G_{p^{2k}}(p^{n+\nu})} {p^{2k (s+\alpha)+c(2-2\alpha-2s)+\frac 32 n-n(s+\alpha)}}\\
			=\left(1-\frac{1}{p^2}\right) 
			\sum_{n=0}^\infty \frac{\tau_{(A')_s\cup\{-\alpha-s\}}(p^{2n})a_{p^{n+\nu}}  } {p^{n}}
		\end{multline}
		while if $\nu>0$ is odd,
		\begin{multline} \label{eqn:identity4} \frac{1}{p^{\nu-\nu (s+\alpha)}} \frac{1-p^{-2(s+\alpha)}}{1-p^{2(\alpha+s)-1}}\left(1-\frac 1 p\right)
			\sum_{k,n=0}^\infty \frac{\tau_{A_s}(p^n)G_{p^{2k}}(p^{n+\nu})}{p^{2k (s+\alpha)+\frac 3 2 n-n (s+\alpha) }}\\
			=\frac 1 {\sqrt{p}}\left(1-\frac{1}{p^2}\right) 
			\sum_{n=0}^\infty \frac{\tau_{(A')_s\cup\{-\alpha-s\}}(p^{2n+1})a_{p^{n+\nu}} }{p^{n}}.
		\end{multline}
		(In fact \eqref{eqn:identity3} is true for even \emph{and} odd $\nu$ and contains \eqref{eqn:identity4} as a special case, but it will be more transparent to separate the two cases.)
		
		By using \eqref{eqn:euler2} to analyze \eqref{eq:back_to_U(s,w)}, the left hand sides of \eqref{eqn:identity3} and \eqref{eqn:identity4} are \emph{formally} the local factor of an Euler product for the left hand side of \eqref{eq:matching_eulers}, and likewise the right hand sides are formally the Euler product for the right hand side of \eqref{eq:matching_eulers}, where $\nu = \nu_p(\ell)$. (Note that the product over odd primes will in each case formally be $2$ times each side of \eqref{eq:matching_eulers}, owing the factors $\lambda = 1/2$ we have labelled.) The products of these local factors over all odd primes will not converge, but from the preceding discussion one sees that if these local factors are multiplied (on both the left and right hand side) by the local factors of $(\zeta^{[2]}(1-2s-2\alpha) Z_{A'}^{[2]}(1-\alpha))^{-1}$, the resulting product over all odd primes will converge, for $\Re s = a$. We therefore will have proved the Lemma if we verify \eqref{eqn:identity3} and \eqref{eqn:identity4}. 
		
		It is sufficient to prove the identities \eqref{eqn:identity3} and \eqref{eqn:identity4} in the case that $s=0$, as the general case follows from this case 
		by replacing the arbitrary set $A$ by $A_s$; therefore we assume that $s=0$ below. 
		
		To prove the identities we let $X=p^{-\alpha}$ and $Y=p^{\alpha-\frac 32}  = \frac{1}{p^{\frac 32}X}$.
		and   observe that 
		\begin{eqnarray*}
			G_{p^{2k}}(p^{m}) =\left\{ \begin{array}{ll}
				\phi(p^m) & \mbox{if $m\le 2k $ and $m$  is even} \\
				p^{2k+1/2} &  \mbox{if $m=2k+1$}\\
				0 & \mbox{otherwise} \end{array} \right.
		\end{eqnarray*}
		so that 
		\begin{eqnarray*}
			\sum_{k=0}^\infty  G_{p^{2k}}(p^{m})X^{2k}=\left\{ \begin{array}{ll}
				\frac{\phi(p^m)}{(1-X^2)} X^m & \mbox{if $m$ is even} \\
				\sqrt{p}   p^{m-1}X^{m-1} &  \mbox{if $m$ is odd} \end{array} \right. 
		\end{eqnarray*}
		From this, we see that $\nu $ even gives 
		\begin{eqnarray*} \frac{1}{p^{\nu+\nu \alpha}}
			\sum_{k,n=0}^\infty 
			\tau_A(p^n)G_{p^{2k}}(p^{n+\nu}) X^{2k}Y^n=
			\sum_{n=0}^\infty 
			\frac{\tau_A(p^{2n})\phi(p^{2n+\nu})}{(1-X^2)p^\nu} X^{2n}Y^{2n}+\sqrt{p}\sum_{n=0}^\infty     \tau_A(p^{2n+1}) p^{2n}X^{2n}Y^{2n+1} .
		\end{eqnarray*}
		
		If $\nu$ is odd we have 
		\begin{eqnarray*} \frac{1}{p^{\nu+\nu \alpha}}
			\sum_{k,n=0}^\infty 
			\tau_A(p^n)G_{p^{2k}}(p^{n+\nu}) X^{2k}Y^n=
			\sum_{n=0}^\infty     \tau_A(p^{2n}) p^{2n-\frac 12}X^{2n-1}Y^{2n}+\sum_{n=0}^\infty 
			\frac{\tau_A(p^{2n+1})\phi(p^{2n+1+\nu})}{(1-X^2)p^\nu} X^{2n+1}Y^{2n+1} .
		\end{eqnarray*}
		
		The first equation, with $\nu$ even, may be rewritten as
		\begin{eqnarray*} \frac{1}{p^{\nu+\nu \alpha}}
			\sum_{k,n=0}^\infty 
			\tau_A(p^n)G_{p^{2k}}(p^{n+\nu}) X^{2k}Y^n=
			\sum_{n=0}^\infty 
			\frac{\tau_A(p^{2n})\phi(p^{2n+\nu})/p^{2n+\nu}}{(1-X^2)p^n} +\frac 1 {pX}\sum_{n=0}^\infty   \frac{  \tau_A(p^{2n+1}) }{p^n}.
		\end{eqnarray*}
		If $\nu=0$ this is 
		\begin{eqnarray*}
			\sum_{k,n=0}^\infty 
			\tau_A(p^n)G_{2k}(p^n) X^{2k}Y^n&=&
			\frac{1}{(1-X^2)}+\frac{\tau_A(p)}{pX}+\sum_{n=1}^\infty\left(\frac{(1-1/p)}{(1-X^2)}\tau_A(p^{2n}) +\frac{\tau_A(p^{2n+1})}{pX}\right)p^{-n}
		\end{eqnarray*}
		If $\nu>0$ this is 
		\begin{eqnarray*} \frac{1}{p^{\nu+\nu \alpha}}
			\sum_{k,n=0}^\infty 
			\tau_A(p^n)G_{p^{2k}}(p^{n+\nu}) X^{2k}Y^n= \frac{(1-1/p)}{(1-X^2)}
			\sum_{n=0}^\infty 
			\frac{\tau_A(p^{2n})} {p^n} +\frac 1 {pX}\sum_{n=0}^\infty   \frac{  \tau_A(p^{2n+1}) }{p^n}.
		\end{eqnarray*}
		The second equation, with $\nu$ odd becomes 
		\begin{eqnarray*} \frac{1}{p^{\nu+\nu \alpha}}
			\sum_{k,n=0}^\infty 
			\tau_A(p^n)G_{p^{2k}}(p^{n+\nu}) X^{2k}Y^n= \frac{1}{\sqrt{p}X}
			\sum_{n=0}^\infty    \frac{ \tau_A(p^{2n}) }{p^n} +\frac1{\sqrt{p}}\frac{(1-\frac 1 p )}{(1-X^2)}\sum_{n=0}^\infty 
			\frac{\tau_A(p^{2n+1}) }{p^n} .
		\end{eqnarray*}
		Let's focus on the $\nu=0$ case. Let $z=1/p$. We need to add in the contribution when $c=1$ which is 
		$$\frac{-1}{p^{2-2\alpha}}\sum_{k=0}^\infty \frac{1}{p^{2k\alpha}}=-\frac{z^2}{X^2(1-X^2)}
		$$
		since $(1-p^{-2\alpha})=1-X^2$.
		Let 
		$$T(n):= \tau_A(p^n) \qquad \mbox{and} \qquad U(n):= \tau_{A'\cup\{-\alpha\}}(p^n).$$
		Then, our identity in the case $\nu=0$ may be rewritten as 
		\begin{eqnarray*}&&
			-\frac{z^2}{X^2} +1+(1-z) \sum_{k=1}^\infty T(2k) z^k+\left(\frac 1 X -X\right)z\sum_{k=0}^\infty T(2k+1)z^k\\
			&&\qquad \stackrel{?}{=} (1+z)\left(1-\frac z {X^2}\right)\left(1+\frac{1}{1+z}\sum_{k=1}^\infty U(2k) z^k\right)
		\end{eqnarray*}
		Let us consider the coefficient of $z^n$ on both sides. For $n=0, 1$ or 2 we can check that the coefficients of $z^n$ on both sides are equal. For $n\ge 3$ we would like to show that
		\begin{eqnarray}   \label{eqn:identity0}
			&&  T(2n)  -T(2n-2) +\left(\frac 1X-X\right) T(2n-1)   \stackrel{?}{=}     
			U(2n) - \frac {U(2n-2)}{X^2}  
		\end{eqnarray}
		For a set $A$ containing an element $\alpha$ recall the identity
		$$\tau_A(p^n)=\tau_{A'}(p^n)+p^{-\alpha}\tau_A(p^{n-1})$$
		where $A'=A-\{\alpha\}$.   
		Let $$t_n=\tau_{A'}(p^n).$$
		Then we have 
		$$T(2n)=t_{2n}+X t_{2n-1}+X^2 T(2n-2);$$
		$$T_{2n-1}=t_{2n-1}+XT_{2n-2};$$
		$$U(2n)=t_{2n}+\frac 1 X t_{2n-1} +\frac{1}{X^2} U(2n-2).$$
		Substituting these into (\ref{eqn:identity0}) we get our desired  identity when $\nu=0$.
		
		When $\nu>0$ is even we have to prove 
		\begin{eqnarray*} (1-z) 
			\sum_{n=0}^\infty 
			T(2n)z^n  + z \left(\frac 1 X -X\right)\sum_{n=0}^\infty  T(2n+1)z^n 
			=\left(1-\frac{z}{X^2}\right)\sum_{n=0}^\infty U(2n)z^n
		\end{eqnarray*}
		This boils down to the same identity (\ref{eqn:identity0}) we just proved.
		
		Finally, when $\nu$ is odd we have to prove
		\begin{eqnarray*} \left(\frac 1 X-X\right)  
			\sum_{n=0}^\infty   T(2n)z^n  +(1-z) \sum_{n=0}^\infty 
			T(2n+1) z^n = \left(1-\frac{z}{X^2}\right)\sum_{n=0}^\infty U(2n+1) z^n
		\end{eqnarray*}
		i.e. that
		$$\left(\frac 1X-X\right) T(2n)+T(2n+1)-T(2n-1) =U(2n+1)-\frac{U(2n-1)}{X^2}.$$
		Again this is a simple change of variables from (\ref{eqn:identity0}).
	\end{proof}
	
	Thus applying Lemmas \ref{lem:func_eq} and \ref{lem:matching_eulers} to \eqref{eq:S_on_line_a} we have shown for sufficiently small positive $a'$ (depending on $\epsilon$),
	
	\begin{multline}
		\label{eq:S_on_line_a with Btilde}
		\mathcal S_A^{\ne 0}(D;\ell) = \sum_{\alpha\in A}  \frac 1{2\pi i} \int_{(a')}    \breve{W}(s)\breve{\Psi}(1-s-\alpha) N^s  8^{-s-\alpha} \chi(1/2+s+\alpha) \frac{\tilde{\mathcal{B}}^{(2)}((A')_s \cup \{-\alpha-s\};\ell)}{2\zeta^{[2]}(2)} D^{1-s-\alpha} ~ds\\ +O(N^{1/2}D^{\epsilon}Y^{1+\epsilon}\ell^{\frac 12+\epsilon} + D^{1+\epsilon} \ell^\epsilon /Y^{1-\epsilon}),
	\end{multline}

	\subsection{Extracting an average in $d$}
	
	Finally we will show that the right hand side of \eqref{eq:S_on_line_a with Btilde} can be replaced by the expression we began with, appearing on the right hand side of Theorem \ref{thm:main}. Theorem \ref{thm:main} involves an average over $d$ of the sums $\mathcal{B}^{(2d)}$, whereas what we have proved in \eqref{eq:S_on_line_a with Btilde} involves the function $\tilde{\mathcal{B}}^{(2)}$, in which the sums defining $\mathcal{B}^{(2)}$ have been multiplied by the factors $a_{n\ell}$. The passage to Theorem \ref{thm:main} from \eqref{eq:S_on_line_a with Btilde} is motivated by an identity,
	\begin{equation}
		\label{eq:an_explanation}
		\sum_{(d,m)=1} \mu^2(2d) \Psi(d/D) d^{-\upsilon} \sim \frac{1}{2\zeta^{[2]}(2)} a_m \breve{\Psi}(1-\upsilon) D^{1-\upsilon},
	\end{equation}
	as $D\rightarrow\infty$, which holds for $m$ odd and $\upsilon$ with small real part. (We will use \eqref{eq:an_explanation} only to motivate the discussion that follows and so we will neither prove it nor specify how small $\upsilon$ must be.) Heuristically, one may pass between Theorem \ref{thm:main} and \eqref{eq:S_on_line_a with Btilde} by applying \eqref{eq:an_explanation} with $m = n\ell$ and $\upsilon = s+\alpha$. Unfortunately the series defining $\mathcal{B}^{(2d)}((A')_s \cup \{\alpha - s\};\ell)$ and $\tilde{\mathcal{B}}^{(2d)}((A')_s \cup \{\alpha - s\};\ell)$ will not converge absolutely for variables $s$ and sets $A$ in the relevant range for Theorem \ref{thm:main} and \eqref{eq:S_on_line_a with Btilde}. But we have the following result, which may be thought of as expressing the same idea: for fixed small $a' >0$ and $\Re s = a' > 0$,
	
	\begin{multline}
		\label{eq:B_to_Btilde}
		\sum_{(d,\ell)=1} \mu^2(2d) \Psi(d/D) \sum_{\alpha \in A} (8d)^{-s-\alpha} \chi(1/2+s+\alpha) \mathcal{B}^{(2d)}((A')_s \cup \{-\alpha-s\};\ell) \\
		= \sum_{\alpha \in A} \breve{\Psi}(1-s-\alpha) 8^{-s-\alpha} \chi(1/2+s+\alpha) \frac{\tilde{\mathcal{B}}^{(2)}((A')_s \cup \{-\alpha-s\};\ell)}{2\zeta^{[2]}(2)}  D^{1-s-\alpha} \\
		+ O_\epsilon(D^{1/2-a'+\epsilon}),
	\end{multline}
	for any $\epsilon > 0$, uniformly for all odd $\ell$.
	
	The reader may have noted that the left hand side and right hand side of \eqref{eq:B_to_Btilde} are linear combinations of simpler expressions, and it is tempting to prove this result by removing the sum over $\alpha \in A$ and treating each piece individually. This would work, but for the possibility that some $\alpha \in A$ are very close in which case the individual terms $\mathcal{B}((A')_s \cup \{-\alpha-s\})$ will have singularities. It is only by summing them together that these singularities cancel out and a result of this sort is proved for all possible sets $A$ with the error term that is claimed.
	
	Let use proceed to the proof of \eqref{eq:B_to_Btilde} which will occupy the rest of this subsection. Given a finite (multi)set $B = \{\beta_1,...,\beta_k\}$, we define the function $\tilde{\mathcal{B}}^{(2)}_{(w)}(B;\ell) = \tilde{\mathcal{B}}^{(2)}_{(w)}(\beta_1,...,\beta_k;\ell)$ for $\Re w > 0$ and $\Re \beta_i > 0$ for all $i$ by
	$$
	\tilde{\mathcal{B}}^{(2)}_{(w)}(B;\ell):= \sum_{(n,2)=1 \atop n\ell = \square} \frac{\tau_B(n)}{\sqrt{n}} \prod_{p | n\ell} \frac{p^w}{p^w+1}.
	$$
	Note that if $\Re w > 1$ and $\Re \beta_i > 0$ for all $i$,
	\begin{multline}
		\label{eq:Btilde_generating}
		\sum_{(d,\ell)=1} \frac{\mu^2(2d)}{d^w} \mathcal{B}^{(2d)}(B;\ell) = \sum_{(d,\ell)=1} \frac{\mu^2(2d)}{d^w} \sum_{(n,2d)=1 \atop n\ell = \square} \frac{\tau_B(n)}{\sqrt{n}} \\
		= \sum_{(n,2)=1 \atop n\ell = \square} \frac{\tau_B(n)}{\sqrt{n}} \sum_{(d,n\ell)=1} \frac{\mu^2(2d)}{d^w} = \tilde{\mathcal{B}}^{(2)}_{(w)}(B;\ell) \frac{\zeta^{[2]}(w)}{\zeta^{[2]}(2w)}.
	\end{multline}
	We will use \eqref{eq:Btilde_generating} to prove \eqref{eq:B_to_Btilde}, but first we must discuss an analytic continuation of $\tilde{B}^{(2)}_{(w)}(B;\ell)$ in the parameters $B$.
	
	Note that for $\Re w > 0$ and $\Re \beta_i > 0$ for all $i$ (and $\ell$ odd),
	\begin{multline*}
		\tilde{\mathcal{B}}_{(w)}^{(2)}(B;\ell) = \prod_{p > 2 \atop p \nmid \ell} \Big( 1+ \sum_{n=1}^\infty \frac{\tau_B(p^{2n})}{p^n} \Big(\frac{p^w}{p^w+1}\Big) \Big) \\
		\times 
		\prod_{p | \ell \atop \nu_p(\ell)\; \mathrm{even}} \Big( \sum_{n=0}^\infty \frac{\tau_B(p^{2n})}{p^n} \Big(\frac{p^w}{p^w+1}\Big) \Big) \prod_{p|\ell \atop \nu_p(\ell)\; \mathrm{odd}} \Big( \sum_{n=0}^\infty \frac{\tau_B(p^{2n+1})}{p^{n+1/2}} \Big(\frac{p^w}{p^w+1}\Big)\Big).
	\end{multline*}
	Likewise note that $\mathcal{B}^{(2)}(B;\ell)$ has the same Euler product in which the factors $\frac{p^w}{p^w+1}$ have been replaced by $1$. Set
	\begin{equation}
		\label{eq:F_def}
		\mathcal{F}_{(w)}(B;\ell) := \frac{\tilde{\mathcal{B}}_{(w)}^{(2)}(B;\ell)}{\mathcal{B}_{(w)}^{(2)}(B;\ell)} = \prod_{p > 2 \atop p \nmid \ell} \frac{\Big( 1+ \sum_{n=1}^\infty \frac{\tau_B(p^{2n})}{p^n} \Big(\frac{p^w}{p^w+1}\Big) \Big)}{\Big( 1+ \sum_{n=1}^\infty \frac{\tau_B(p^{2n})}{p^n} \Big)} \prod_{p | \ell} \frac{p^w}{p^w+1}.
	\end{equation}
	The Euler product defining $\mathcal{F}_{(w)}(B;\ell)$ will converge as long as $\Re \beta_i \geq -1/4+\delta$ for all $i$ and $\Re w \geq 1/2$ for $\delta > 0$ and so we may extend the definition of $\mathcal{F}_{(w)}(B;\ell)$ into this domain. One may check through the Euler product that on this domain for fixed $\delta > 0$,
	\begin{equation}
		\label{eq:F_bound}
		\mathcal{F}_{(w)}(B;\ell) = \prod_{p > 2 \atop p \nmid \ell} \Big( 1 + O(\Big( \frac{1}{p^{1+2\delta}}\Big) \Big) \prod_{p | \ell} \frac{p^w}{p^w+1} \ll 1,
	\end{equation}
	where the bound is uniform in $\ell$.
	
	We return to \eqref{eq:Btilde_generating}; by analytic continuation it will also hold in the domain $\Re w \geq 1/2$ and $\beta_i \geq -1/4+\delta$ for all $\beta_i \in B$.
	
	Hence if $c> 1$, and $\Re \beta_i \geq -1/4+\delta$ for all $i$, and $\upsilon$ is any complex number,
	\begin{multline}
		\label{eq:d_average_to_contour}
		\sum_{(d,\ell)=1} \mu^2(2d) \Psi(d/D) d^{-\upsilon}  \mathcal{B}^{(2d)}(B;\ell) = \frac{1}{2\pi i} \int_{(c)} \breve{\Psi}(z-\upsilon) D^{z-\upsilon} \sum_{(d,\ell)=1} \frac{\mu^2(2d)}{d^z} \mathcal{B}^{(2d)}(B;\ell)\, dz \\
		= \frac{1}{2\pi i} \int_{(c)} \breve{\Psi}(z-\upsilon) D^{z-\upsilon}  \tilde{\mathcal{B}}^{(2)}_{(z)}(B;\ell) \frac{\zeta^{[2]}(z)}{\zeta^{[2]}(2z)}\, dz.
	\end{multline}
	
	Thus putting $\upsilon = s+\alpha$ and $B = (A')_s \cup \{-\alpha-s\}$, then summing over $\alpha \in A$, we see that the left hand side of \eqref{eq:B_to_Btilde} is
	$$
	\frac{1}{2\pi i} \int_{(c)} \Big( \sum_{\alpha \in A} \breve{\Psi}(z-s-\alpha)8^{-s-\alpha}\chi(1/2+s+\alpha) \tilde{\mathcal{B}}^{(2)}_{(z)}((A')_s \cup \{-\alpha-s\};\ell)\Big) \frac{\zeta^{[2]}(z)}{\zeta^{[2]}(2z)}\, dz.
	$$
	
	Push the contour $\Re z = c$ to $\Re z = 1/2$. The residue at $z=1$ straightforwardly gives the main term on the right hand side of \eqref{eq:B_to_Btilde}, while the contour at $\Re z = 1/2$ will give the error term.
	
	To verify the claimed error term, condense the sum in parentheses using Corollary \ref{cor:doublepermutationsum} of the Appendix, exactly as in the proof of Proposition \ref{prop:swap_dropping} earlier. From \eqref{eq:F_def} and \eqref{eq:F_bound}, we have $\tilde{\mathcal{B}}_{(w+\upsilon)}^{(2)}(B;\ell) \ll |\mathcal{B}^{(2)}(B;\ell)|$ as long as $\Re z \geq 1/2$ for index sets $B$ with all elements to the right of the line right real part $-1/4$. The decay of $\breve{\Psi}(z-s-\alpha)$ means the integral will be effectively localized to $z$ close to $s+\alpha$. Use Lemma \ref{lem:B_cont} and the Lindel\"of Hypothesis and Stirling's formula \eqref{eq:stirling} to see that the error term claimed in \eqref{eq:B_to_Btilde} is obtained.

	\subsection{Completing the proof of Theorem \ref{thm:main}}
	
	Applying \eqref{eq:B_to_Btilde} to \eqref{eq:S_on_line_a with Btilde} we have, for $\epsilon >0$ and sufficiently small $a' > 0$ (depending on $\epsilon$),
	\begin{multline*}
		\mathcal{S}_A^{\neq 0}(D;\ell) = \sum_{(d,\ell)=1 } \mu^2(2d) \Psi\left(\frac d{D}\right) 
		\frac{1}{2\pi i}\int_{(a)}\breve W(s) N^s 
		\sum_{\alpha \in A}  X_{8d}(1/2 +s+\alpha) \mathcal B^{(2d)}((A')_s\cup \{-\alpha-s\};\ell) ~ds \\
		+O(D^{1/2+\epsilon} + N^{1/2}D^{\epsilon}Y^{1+\epsilon}\ell^{\frac 12+\epsilon} + D^{1+\epsilon} \ell^\epsilon /Y^{1-\epsilon})
	\end{multline*}
	From \eqref{eq:1swap} we have an analogous expression for $\mathcal{S}_A^0(D;\ell)$. Optimize the parameter $Y$ by setting $Y = D^{1/2}/\ell^{1/4}$ in both cases. The decomposition $\mathcal{S}_A = \mathcal{S}_A^0 + \mathcal{S}_A^{\neq 0}$ then yields the Theorem in the setting that the real part $a'$ of the line of integration is chosen to be sufficiently small based on $\epsilon$. 
	
	To prove the Theorem as stated we must show that this implies the same result when the line of integration is taken to have real part $a=1/10$. But one may shift the contour in this result between a line with sufficiently small real part $a'$ (based on $\epsilon$) and a line with real part $a = 1/10$ since the  functions $\mathcal{B}^{(2d)}((A')_s\cup \{-\alpha-s\};\ell)$ will cross no poles in such a contour shift. 
	This proves the Theorem for small fixed $a$, e.g. $a = 1/10$.
	
	\appendix
	
	\section{Symmetric sums and contour integrals}
	
	In this section we prove a few results which condense permutation sums with singularities into contour integrals. The method dates at least back to \cite[Sec. 2.5]{CFKRS}. We use the following notation:
	$$
	\Delta(\alpha_1,...,\alpha_k) := \prod_{i < j}(\alpha_j-\alpha_i),
	$$
	with the convention for a singleton $\Delta(\alpha_1):=1$. Similarly we set
	$$
	\Delta(z_1,...,z_j;\alpha_1,...,\alpha_k):= \prod_{r=1}^j \prod_{s=1}^k (z_r-\alpha_s).
	$$
	Note that $\Delta(\alpha_1,...,\alpha_k)^2$ is symmetric, and $\Delta(z_1,...,z_j;\alpha_1,...,\alpha_k)$ is symmetric in $z_1,...,z_j$ and $\alpha_1,...,\alpha_k$.
	
	\begin{prop}
		\label{prop:singlepermutationsum}
		Let $A = \{\alpha_1,...,\alpha_k\}$ be a collection of complex numbers, and for $j \leq k$ consider $G: \mathbb{C}^j \rightarrow \mathbb{C}$. Suppose $E$ is an open connected set in $\mathbb{C}$ containing $A$, and $G(z) = G(z_1,...,z_j)$ is analytic in each $z_i$ on the region $E$ (taking all other variables $z_s$ as fixed). Then
		$$
		\frac{1}{(2\pi i)^j} \int_\Gamma \cdots \int_\Gamma \frac{G(z) \Delta(z)^2}{\Delta(z;A)}\, d^j z = \sum_{u_1,...,u_j \in A \atop \textrm{distinct}} \frac{G(u_1,...,u_j)}{\Delta(u; A-u)},
		$$
		where $\Gamma$ is a closed curve or union of closed curves in $E$ which encircles each element of $A$ exactly once.
		
		The right hand side is well-defined if all elements of $A$ are distinct, and if it is interpreted by continuity in $\alpha_1,...,\alpha_k$ if some elements of $A$ coincide, this relation remains true with the right hand side bounded.
	\end{prop}
	
	On the right hand side, we have let $u$ be a shorthand for the (multi)set $\{u_1,...,u_j\}$, so that $\Delta(u;A-u) = \Delta(u_1,..,u_j; v_1,...,v_{k-j})$ where $v_1,..., v_{k-j}$ are the elements of $A$ not in $u$. Note that in the sum on the right hand side $u_1,...,u_j$ can appear in any order, so that if $G$ is a symmetric function in $j$ variables, the left hand side will be
	$$
	j! \sum_{U \subset A\atop |U| = j} \frac{G(U)}{\Delta(U; A-U)}.
	$$
	
	\begin{proof}
		We first suppose all elements of $A$ are distinct. Then the residue theorem gives that the left hand side is
		$$
		= \sum_{u_1,...,u_j \in A} \frac{G(u_1,...,u_j) \Delta(u_1,...,u_j)^2}{\prod_{r=1}^j \prod_{\alpha \in A - \{u_r\}} (u_r-\alpha)}.
		$$
		But if any of $u_1,...,u_j$ coincide, $\Delta(u_1,...,u_j)^2 = 0$. Otherwise
		$$
		\frac{\Delta(u_1,...,u_j)^2}{\prod_{r=1}^j \prod_{\alpha \in A - \{u_r\}} (u_r-\alpha)} = \frac{1}{\Delta(u;A-u)}
		$$
		and the claim is proved.
		
		Now note that the left hand side in the theorem is analytic in the interior of $\Gamma$ for each variable $\alpha_i$. Thus any singularity appearing in the sum of the right hand side must be removable, and the right hand side can be interpreted by continuity and will remain bounded.
	\end{proof}
	
	\begin{corollary}
		\label{cor:doublepermutationsum}
		For $A = \{\alpha_1,...,\alpha_k\}$ a collection of complex numbers, $E$ an open connected set in $\mathbb{C}$ containing $A$, and $\Gamma$ a union of closed curves in $E$ which encircles each $\alpha_i$ exactly once, suppose $F(z;w) = F(z_1,..,z_j; w_1,...,w_{k-j})$ is analytic in each variable in the region $E$, as in Proposition \ref{prop:singlepermutationsum}. Then
		\begin{multline}
			\label{eq:double_contours}
			\frac{1}{(2\pi i)^k} \int_\Gamma \cdots \int_\Gamma  \int_\Gamma \cdots \int_\Gamma \frac{F(z;w) \Delta(z;w) \Delta(z)^2 \Delta(w)^2}{\Delta(z;A) \Delta(w;A)} \, d^{k-j}w \, d^j z \\
			= \sum_{u_1,...,u_j, v_1, ..., v_{k-j} \in A \atop \textrm{distinct}} \frac{F(u;v)}{\Delta(u;v)}.
		\end{multline}
		The right hand side is well-defined if all elements of $A$ are distinct, and if it is interpreted by continuity in $\alpha_1,...,\alpha_k$ if some elements of $A$ coincide, it will remain bounded.
	\end{corollary}
	
	Note that if $F(u;v)$ is a symmetric function in the $\ell$ variables of $u$ and is likewise a symmetric function in the $k-j$ variables of $v$, then the right hand side of \eqref{eq:double_contours} will be
	$$
	= j! \,(k-j)! \sum_{U \subset A \atop |U| = \ell} \frac{F(U; A-U)}{\Delta(U;A-U)}.
	$$
	
	\begin{proof}
		As before we suppose the elements of $A$ are distinct to begin. Apply Proposition \ref{prop:singlepermutationsum} to both the $z$ and $w$ integerals. We have the right hand side of \eqref{eq:double_contours} is
		$$
		= \sum_{u_1,...,u_j \in A \atop \textrm{distinct}} \sum_{v_1,...,v_{k-j} \in A \atop \textrm{distinct}} \frac{F(u;v) \Delta(u;v)}{\Delta(u;A-u)\Delta(v;A-v)}.
		$$
		The term $\Delta(u;v)$ will vanish if $u$ and $v$ share any common elements, so we may suppose that $u$ and $v$ have no elements in common. In this case $\Delta(u;v)/\Delta(u;A-u)\Delta(v;A-v) = 1/\Delta(u;v)$ as claimed.
		
		As in the previous proof, if elements of $A$ are not distinct, because the left hand side is analytic in the interior of $\Gamma$, we may interpret the right hand side by continuity and all singularities will be removable.
	\end{proof}


\begin{thebibliography}{9}
		
		\bibitem[AnKe]{AnKe} Andrade, J. C.; Keating, J. P. The mean value of $L(\frac12,\chi)$ in the hyperelliptic ensemble. J. Number Theory 132 (2012), no. 12, 2793--2816.
		
		\bibitem[BaTu]{BaTu} Siegfred Baluyot, Caroline L. Turnage-Butterbaugh. A mean value theorem for Dirichlet polynomials associated with primitive Dirichlet $L$-functions. Int. Math. Res. Not. IMRN 2025, no. 3, Paper No. rnaf010, 39 pp.
		
		\bibitem[BaCo]{BaCo} Baluyot, Siegfred, and Brian Conrey. Moments of zeta and correlations of divisor-sums: stratification and Vandermonde integrals. Mathematika 70 (2024), no. 2, Paper No. e12243, 32 pp.
		
		\bibitem[BDPW]{BeDiPeWe} Jonas Bergström, Adrian Diaconu, Dan Petersen, Craig Westerland. Hyperelliptic curves, the scanning map, and moments of families of quadratic L-functions. arXiv preprint arXiv:2302.07664 (2024).
		
		\bibitem[\v{C}e]{Ce} \v{C}ech, Martin. Moments of real Dirichlet $L$-functions and multiple Dirichlet series. arXiv preprint arXiv:2402.07473 (2024).
		
		
		\bibitem[CFKRS]{CFKRS} J.B. Conrey, D.W. Farmer,
		J.P. Keating, M.O. Rubinstein and N.C. Snaith.  Integral
		moments of $L$-functions. Proc. Lond. Math. Soc. 91 (2005)
		33--104.
		
		\bibitem[CoFaz]{CoFaz} J.B. Conrey, A. Fazzari. Averages of long Dirichlet polynomials with modular coefficients. Mathematika 69 (2023), no. 4, 1060--1080.
		
		
		
		\bibitem[CoGh]{CoGh} Conrey, J. B.; Ghosh, A. A conjecture for the sixth power moment of the Riemann zeta-function. Internat. Math. Res. Notices 1998, no. 15, 775--780.
		
		\bibitem[CoGo]{CoGo} Conrey, J. B.; Gonek, S. M. High moments of the Riemann zeta-function. Duke Math. J. 107 (2001), no. 3, 577--604.
		
		\bibitem[CoIwSou]{CoIwSou} J. B. Conrey, H. Iwaniec, and K. Soundararajan. The sixth power moment of
		Dirichlet L-functions. Geom. Funct. Anal. 22.5 (2012), pp. 1257–1288.
		
		\bibitem[CoKeI]{CoKeI} Brian Conrey and Jonathan P. Keating. Moments of zeta and correlations of divisor-sums: I. Philos.
		Trans. Roy. Soc. A 373: 20140313 (2015).
		
		\bibitem[CoKeII]{CoKeII} Brian Conrey and Jonathan P. Keating. Moments of zeta and correlations of divisor-sums: II. Advances
		in the Theory of Numbers. Vol. 77. Fields Inst. Commun. Fields Inst. Res. Math. Sci., Toronto, ON, 2015, pp. 75–85.
		
		\bibitem[CoKeIII]{CoKeIII} Brian Conrey and Jonathan P. Keating. Moments of zeta and correlations of divisor-sums: III. Indag.
		Math. (N. S.) 26 (2015), 736–747.
		
		\bibitem[CoKeIV]{CoKeIV} Brian Conrey and Jonathan P. Keating. Moments of zeta and correlations of divisor-sums: IV. Res.
		Number Theory 2, Article number: 24 (2016).
		
		\bibitem[CoKeV]{CoKeV} Brian Conrey and Jonathan P. Keating. Moments of zeta and correlations of divisor-sums: V. Proc.
		London Math. Soc. (3) 118 (2019), 729–752.
		
		\bibitem[ChLi]{ChLi} Vorrapan Chandee and Xiannan Li. The eighth moment of Dirichlet L-functions.
		Adv. Math. 259 (2014), pp. 339–375.
		
		\bibitem[DiWh]{DiWh} Diaconu, Adrian; Whitehead, Ian. On the third moment of $L(\frac{1}{2}, \chi_d)$ II: the number field case. J. Eur. Math. Soc. (JEMS) 23 (2021), no. 6, 2051--2070.
		
		\bibitem[Di]{Di} Diaconu, Adrian. On the third moment of $L(\frac{1}{2},\chi_d)$ I: The rational function field case. J. Number Theory 198 (2019), 1--42.  
		
		\bibitem[Dj]{Dj} Djanković, Goran. The mixed second moment of quadratic Dirichlet $L$-functions over function fields. Rocky Mountain J. Math. 51 (2021), no. 6, 2003--2017.
		
		\bibitem[Flo1]{Flo1} Alexandra Florea. Moments and zeros of L-functions over function fields. Thesis, Stanford University. 2017,  179 pages.
		
		\bibitem[Flo2]{Flo2} Alexandra Florea. The fourth moment of quadratic L-functions over function fields.  Geom. Funct. Anal. 27 (2017), no. 3, 541--595.
		
		\bibitem[Flo3]{Flo3} Florea, Alexandra M. Improving the error term in the mean value of $L(\frac12,\chi)$ in the hyperelliptic ensemble. Int. Math. Res. Not. IMRN 2017, no. 20, 6119--6148.
		
		\bibitem[Flo4]{Flo4} Florea, Alexandra. The second and third moment of $L(1/2,\chi)$ in the hyperelliptic ensemble. Forum Math. 29 (2017), no. 4, 873--892.
		
		\bibitem[GG]{GG}
		D.A.  Goldston and S.M.  Gonek.  Mean value theorems for long Dirichlet polynomials and tails of 
		Dirichlet series. Acta Arith. 84 (1998), no. 2, 155--192.
		
		\bibitem[HG]{HaNg}
		A. Hamieh and N. Ng. Mean values of long Dirichlet polynomials with higher divisor coefficients. Adv. Math. 410 (2022), part B, Paper No. 108759, 61 pp.
		
		\bibitem[IR]{IrRo} 
		K. Ireland and M. Rosen. A classical introduction to modern number theory. Second edition. Graduate Texts in Mathematics, 84. Springer-Verlag, New York, 1990.
		
		\bibitem[KoMiVa]{KoMiVa}Kowalski, E., Michel, P., VanderKam, J. M. Mollification of the fourth moment of automorphic $L$-functions and arithmetic applications. Invent. Math. 142 (2000), no. 1, 95--151.
		
		\bibitem[MPPRW]{MiPaPeRaWi} Jeremy Miller, Peter Patzt, Dan Petersen, Oscar Randal-Williams. Uniform twisted homological stability. arXiv preprint arXiv:2402.00354 (2024).
		
		\bibitem[MV]{MV}
		H.L Montgomery and R.C. Vaughan. Multiplicative number theory. I. Classical theory. Cambridge Studies in Advanced Mathematics, 97. Cambridge University Press, Cambridge, 2007. 
		
		\bibitem[RoSou]{RoSou} Rodgers, Brad; Soundararajan, Kannan. The variance of divisor sums in arithmetic progressions. Forum Math. 30 (2018), no. 2, 269--293.
		
		\bibitem[Son]{Son} Sono, Keiju. The second moment of quadratic Dirichlet $L$-functions. J. Number Theory 206 (2020), 194--230.
		
		\bibitem[Sou]{Sou}
		K. Soundararajan.  Non-vanishing of
		quadratic Dirichlet $L$-functions at $s=\frac12$,    Ann. of Math.
		(2)   152   (2000)  pp. 447--488.
		
		
		
		
	\end{thebibliography}
\end{document}